\documentclass[letter,11pt]{amsart}
\usepackage[toc,page]{appendix}
\usepackage{amssymb} 
\usepackage{bbm}
\usepackage{graphicx}
\usepackage{color}
\usepackage[font=footnotesize]{caption}
\usepackage{mathtools}
\usepackage{pinlabel}
\usepackage{graphicx}
\graphicspath{ {images/} }
\usepackage{subfig}
\usepackage{comment}
\usepackage{multirow}
\usepackage{longtable}

\usepackage{hyperref}

\usepackage{float}
\usepackage{calc}
\def\thetitle{{ . }}
\hypersetup{
pdftitle=  \thetitle,
pdfauthor=  {}  
}

\newtheorem{thm}{Theorem}[section]
\newtheorem{lem}[thm]{Lemma}
\newtheorem{cor}[thm]{Corollary}
\newtheorem{prop}[thm]{Proposition}
\newtheorem{que}{Question}
\newtheorem{conj}{Conjecture}
\newtheorem{defn}[thm]{Definition}

\theoremstyle{remark}

\newtheorem*{rmk}{\textbf{Remark}}

\usepackage{tikz}
\usetikzlibrary{calc,decorations.markings}
\usetikzlibrary{shapes,snakes}

\theoremstyle{definition}
\newtheorem*{defn*}{Definition}
\newtheorem*{que*}{Question}

\newcommand\R{\mathbb{R}}
\newcommand\C{\mathbb{C}}
\newcommand\Z{\mathbb{Z}}

\newcommand\Hom{\operatorname{Hom}}

\newcommand\torsion{\operatorname{torsion}}
\DeclarePairedDelimiter\floor{\lfloor}{\rfloor}

\title[Thurston unit ball of a family of $n$-chained links]{Thurston unit ball of a family of $n$-chained links and their fibered face}

\author{Juhun Baik}
\address{Juhun Baik, Department of Mathematical Sciences, KAIST,  
291 Daehak-ro, Yuseong-gu, Daejeon 34141, South Korea }
\email{jhbaik@kaist.ac.kr}

\author{Philippe Tranchida}

\address{Philippe Tranchida, Universié Libre de Bruxelles, Département de Matématiques, C.P.216-Algèbre et Combinatoire, Boulevard du Triomphe, 1050 Bruxelles, Belgium }
\email{tranchida.philippe@gmail.com}

\date{\today}

\begin{document}
\maketitle

\begin{abstract}
We determine the Thurston unit ball of a family of $n$-chained link, denoted by $C(n,p)$, where $n$ is the number of link components and $p$ is the number of twists.
When $p$ is strictly positive, we prove that the Thurston unit ball for $C(n,p)$ is an $n$-dimensional cocube, for arbitrary $n$.
Moreover, we clarify the condition for which $C(n,p)$ is fibered and find at least one fibered face for any $p$.
Finally we provide the Teichm\"uller polynomial for the face of Thurston unit ball of $C(n, -2)$ with $n\geq 3$.
\end{abstract}

\section{Introduction}\label{sec:intro}

Let $M$ be a $3$-dimensional manifold. We will suppose that $M$ has tori boundaries for simplicity. 
In one of his many seminal works \cite{william1986norm}, W. Thurston introduced a notion of a semi-norm for the second homology vector spaces of $M$. 
More precisely, let $[a] \in H_2(M, \partial M ; \mathbb{Z})$ be an integral second homology class. 
Then $[a]$ can be represented by a disjoint union of properly embedded surfaces $S_i$.
The Thurston norm of $[a]$ is then defined to be
$$
x(a) := \min \{\sum_i \max \{0,-\chi(S_i) \}\}
$$
where the minimum is taken over all possible ways to represent $[a]$ as a disjoint union of properly embedded surfaces. 
If $M$ is irreducible and atoroidal, this then extends to a norm on $H_2(M, \mathbb{R})$.
We sometimes use $||\cdot||$ to denote the Thurston norm.
In the same paper, he proves that the unit ball with respect to that norm, that we will call Thurston unit ball, is always a polytope.
Even though this concept has had huge theoretical consequences, it seems that there are very few cases for which we know precisely what the Thurston unit ball is. 
An interesting question in that regard is the following.

\begin{que*}
    Which polytope can appear as a Thurston unit ball of some $3$-manifold?
\end{que*}

This question was already posed by Kitayama in \cite{kitayama2022survey}.
It has been generalized in terms of groups and their $1$st homology by Friedl, Lück and Tillmann \cite{friedl2016groups}.
In \cite{pacheco2019thurston}, Pacheco-Tallaj, Schreve and Vlamis find out the shape of the Thurston unit ball for tunnel number-one manifolds.
We refer \cite{kitayama2022survey} for more recent research of Thurston norm.

In this article we show that the Thurston unit ball can contain highly symmetric polytopes in arbitrary high dimensions.
We will be interested in determining the Thurston unit ball for a family of complements of links, denoted by $C(n,p)$. 
Briefly speaking, $C(n,p)$ is an $n$-chained link with $p$ positive half-twist on the first component if $p$ is positive or $p$ negative half-twist on the first component if $p$ is negative (see figure \ref{fig:C(n,p)ex}). 
The complements of these links are in some sense generalizations of the magic manifold, which is the complement of $C(3,0)$. 
The magic manifold and its properties are thus good examples to keep in mind.

In a previous article \cite{baik2022topological}, the two authors together with Harry Baik and Changsub Kim studied the relation between the minimal entropy of pseudo-anosov maps on a surface $S$ and the action of these maps on $H_1(S)$. 
In order to do so, the use of the complement of $C(n,-2)$ was crucial.

Here are the main results of this article.

\begin{thm}
Let $M(n,p)$ be the complement of the link $C(n,p)$ with $n \geq 3$ and $B(n,p)$ be the Thurston unit ball of $M(n,p)$. 
Suppose $M(n, p)$ is hyperbolic.
Then
\begin{itemize}
    \item If $p \ge 1$, $B(n,p)$ is an $n$-dimensional cocube with vertices \\
    $(\pm 1, 0, \cdots, 0), \cdots, (0, \cdots, 0, \pm 1)$. (Corollary \ref{cor:normballp})
    \item If $p = 0$, $B(n,p)$ is the union of an $n$-dimensional cocube and two simplices. (Theorem \ref{thm:normball0})
\end{itemize}
\end{thm}
Moreover, we also find a fibered face in each case, and determine the topological type of every fibered surface in that fibered face.

A complete answer for the case of $p < 0$ is out of our reach for now. We nonetheless find a set $V(n,p)$ of points in the Thurston unit ball and conjecture that their convex hull, denoted by $B(n,p)$ is the whole Thurston unit ball. This conjecture is partially supported by computational data, obtained using the program Tnorm and gathered in Appendix \ref{appendix:B}. 

\begin{conj}
$B(n,p)$ is equal to the Thurston unit norm ball of $C(n, p)$ when $p<0$.
\end{conj}
For the case of $p = -2$, we also managed to compute the Teichmuller polynomial for every value of $n$. 

\begin{thm} (Theorem \ref{thm:teichpoly})
    The Teichm\"{u}ller polynomial $P$ for the fibered face $\mathcal{F}$ is
    \[
        P(x_1, \cdots, x_{n-1}, u) := A - \sum_{k = 1}^n ua_kA_k
    \]
     where $a_1 = 1, a_2 = x_1^{-1}, \cdots, a_n = (x_1\cdots x_{n-1})^{-1}$, $A := (a_1 - u)\cdots(a_n - u)$ and $A_k = \dfrac{A}{(a_k - u)(a_{k-1}-u)}$, subscript $k \equiv (mod~n) + 1$.
\end{thm}

\subsection*{Acknowledgement}

We would like to thank Harry Hyungryul Baik, Chenxi Wu, Dan Margalit, JungHwan Park and Livio Liechti for useful conversations. We also thank William Worden for helping us setting up and using the Tnorm package.

\section{Preliminary}\label{sec:prelim}

We gather here the essential tools that will be used in the rest of the paper.

\subsection{Murasugi sums}
In his papers \cite{gabai1985murasugi}, \cite{gabai1986genera}, 
David Gabai proved theorems related to the fiberedness of embedded surfaces and about their monodromy map when they are in fact fibered. He constantly makes use of a geometric operation called ``Murasugi sum", which is a way to glue surfaces together while preserving some property of their fibers. 
We begin with the definition of this Murasugi sum.
\begin{defn}[Murasugi sum, \cite{gabai1985murasugi}]
    The oriented surface $\Sigma \subset S^3$ is a \emph{Murasugi sum} of two different oriented surface $\Sigma_1$ and $\Sigma_2$ if 
    \begin{enumerate}
        \item $\Sigma = \Sigma_1 \cup \Sigma_2$ and $\Sigma_1 \cap \Sigma_2 = D$, where $D$ is a $2n$-gon.
        \item There is a partition of $S^3$ into two $3$-balls $B_1, B_2$ satisfying that 
        \begin{itemize}
            \item $\Sigma_i \subset B_i$ for $i = 1,2$.
            \item $B_1\cap B_2 = S^2$ and $\Sigma_i \cap S^2 = D$ for $i = 1,2$.
        \end{itemize}
    \end{enumerate} 
\end{defn}
In simple terms, the Murasugi sum is a way to cut-and-paste two surfaces in an alternating way so that, around the gluing region, it looks like there are $2n$ legs going up and down alternatively.

\begin{figure}
	\centering
	\includegraphics[width = 0.8\textwidth]{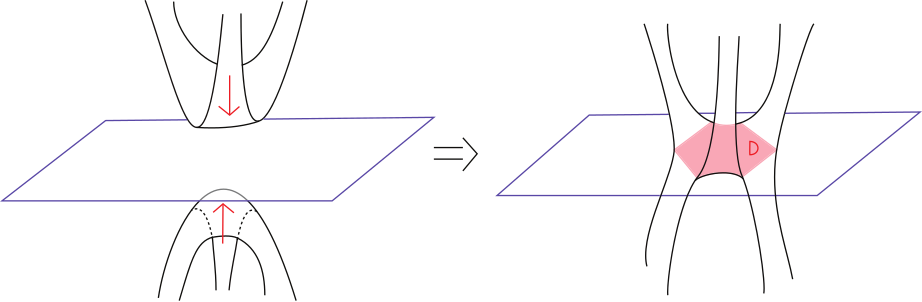}
	\caption{Murasugi sum of two surfaces, here $D$ is a hexagon}
	\label{fig:murasugi_sum}
\end{figure}

The power of the Murasugi sum is that it preserves the fiberedness and also the monodromies. More precisely, Gabai proved the two following theorems.

\begin{thm}[Gabai, \cite{gabai1983murasugi}]
    Let $S$ be a Murasugi sum of $S_1$ and $S_2$.
    Then $S$ is a fibered surface if and only if both $S_1$ and $S_2$ are fibered surfaces.
\end{thm}
\begin{thm}[\cite{gabai1985murasugi}, Cor 1.4]
    Suppose that $R$ is a Murasugi sum of $R_1, R_2$ with $\partial R_i = L_i$, where $L_i$ is a fibered link with monodromy $f_i$ fixing pointwise the boundary $\partial R_i$, resp.
    Then $L = L_1\cup L_2$ is a fibered link with fiber $R$ and its monodromy map is $f = f_2'\circ f_1$ where $f_i'$ is the map induced on $R$ by inclusion.
\end{thm}

Using two theorems it is possible to construct fibered surfaces by gluing together smaller fibered surfaces while having a nice control on the monodromy maps. 
A good starting block for this construction is the Hopf link, which consists of $2$ circles linked together exactly once.
The Hopf band is then a Seifert surface of the Hopf link. 
It is thus a fibered surface of $S^3 - \{\text{Hopf link}\}$.
\begin{lem}[Monodromy of a Hopf band]
    The Hopf band is a fibered surface.
    Moreover, the monodromy of the positive (resp. negative) Hopf band is the right-handed (resp. left-handed) Dehn twist along its core curve.
\end{lem}
In fact, in \cite{giroux2006stable} Giroux and Goodman proved that every fibered link in $S^3$ is obtained from the unknot, by Murasugi summing or desumming along Hopf bands.
\emph{i.e.,} The Hopf bands are building blocks to construct any fibered link in $S^3$.

\subsection{Fibers of alternating knots/links}
Suppose a fibered link $L$ in $S^3$ is given.
In general it is really hard to detect what is the fiber of $S^3 - L$.
However, if $L$ is alternating, Seifert showed in \cite{seifert1935geschlecht} how to construct the fiber. 
We first recall the definition of an alternating link and then explain the Seifert algorithm.
For more details, we refer \cite{rolfsen2003knots}.

\begin{defn}[Alternating link]
    Let $L$ be an oriented link.
    An alternating diagram for $L$ is a link diagram such that the crossings alternate under and over as one travels along each component of the link. 
    A link is alternating if it admits an alternating diagram.
\end{defn}

\begin{defn}[Seifert algorithm]
    Let $L$ is an oriented link.
    The Seifert algorithm can be described as follows.
    \begin{enumerate}
        \item At each crossing, cut at the crossing and paste back in such a way that, near the crossing, there are $2$ components, as showed in figure \ref{fig:cutPaste}.
        
        \begin{figure}[h] 
            \centering
            \includegraphics[width=.5\textwidth]{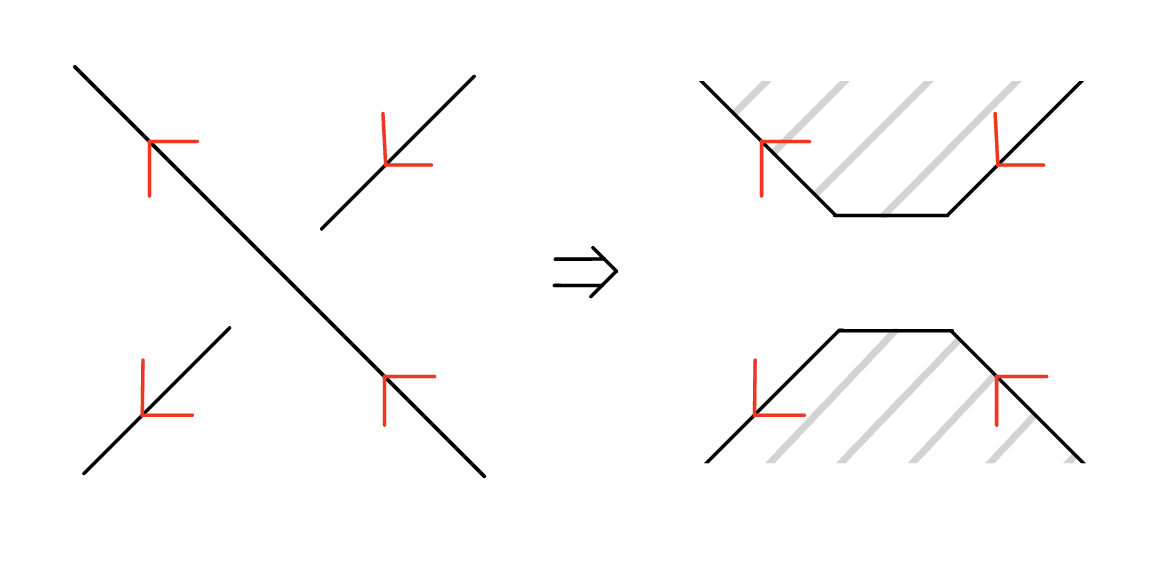}
            \caption{Cut and paste at a crossing in Seifeirt algorithm}
            \label{fig:cutPaste}
        \end{figure}        
        \item After all these cut-and-paste operations, we are left with a disjoint collection of oriented simple closed curves.
        Each curve bounds a disk, unless two or more curves are nested.
        In that case, we can consider the innermost to be lying slightly above the others and repeat this process until there are no more nested curves.
        We then assign to each region "+" sign if the region is on the left side of the boundary curve, with respect to its orientation, or "-" sign otherwise.
        Note that the result is sometimes called a checkerboard coloring.
        \item Finally, we reconnect these discs at each crossing with a twisted strip, the direction of the twist being determined by the direction of the original crossing.
    \end{enumerate}
    The result of this algorithm is a surface $S$ whose (oriented) boundary is $L$.
\end{defn}
The surface obtained from the Seifert algorithm is called the \textit{Seifert surface} of $L$.
The \textit{genus} of a link $L$ is defined to be the minimal genus of a surface in the complement of $L$ whose boundary is $L$.
In \cite{gabai1986genera} Gabai proved that if $L$ is alternating, then the genus of $L$ is equal to the genus of the Seifert surface of $L$.
\begin{thm}[\cite{gabai1986genera}, Thm 4]
    Let $L$ be an oriented link in $S^3$.
    If $S$ is a surface obtained by applying Seifert's algorithm to an alternating diagram of $L$, then $S$ is a surface of minimal genus.
\end{thm}
This condition is closely related to the fiberedness.
\begin{thm}[Theorem 4.1.10 in \cite{kawauchi1996survey}]
    Let $S$ be a Seifert surface for a fibered link $L$.
    Then the following are equivalent.
    \begin{enumerate}
        \item $S$ attains the minimal Seifert genus.
        \item $S$ is a fibered surface.
    \end{enumerate}
\end{thm}

\subsection{Teichm\"uller polynomial} \label{sec:teich}

The \textit{Teichm\"uller polynomial} $\theta_F$ for a fibered face $\mathcal{F} \subset H^1(M,\mathbb{R})$ is a polynomial invariant that determines the stretch factors of all the monodromy of fibers in $\mathcal{F}$. Similarly to the Alexander polynomial, the Teichm\"uller polynomial has coefficients in the group ring $\mathbb{Z}(G)$ where $G = H_1(M,\mathbb{Z})/$torsion.

We describe here one way to compute the Teichm\"uller polynomial. Let $\varphi \colon S \to S$ be a pseudo-Anosov mapping and let $x = x_1, \cdots x_{n-1}$ be a multiplicative basis for 

$$  
H = \Hom(H^1(M,\mathbb{Z}^\varphi),\mathbb{Z})
$$
, where $H^1(M,\mathbb{Z})^\varphi$ is the $\varphi$-invariant cohomology.
Remark that we can construct a natural map from $\pi_1(S)$ to $H$ by evaluating cohomology classes on loops. Choose now a lift $\Tilde{\varphi} \colon \Tilde{S} \to \Tilde{S}$ of $\varphi$ to the cover $\Tilde{S}$ corresponding to $H$ under the previous map.

Let $M = S \times [0,1]/ (p,1) \sim (\varphi(p),0)$ be the mapping torus of $\varphi$. Then we have that

$$ G = H_1(M,\mathbb{Z})/ \torsion = H \oplus \mathbb{Z}$$

We let $u$ denote the generator of the $\mathbb{Z}$ component of $G$ so that $G$ is generated by $x_1,\cdots,x_{n-1}$ and $u$.
Let $V$ and $E$ be the vertices and the edges of an invariant train track $\tau$ on $S$ carrying the pseudo-Anosov map $\varphi$. The lifts $\Tilde{V}$ and $\Tilde{E}$ of $V$ and $E$ to $\Tilde{S}$ can respectivily be considered as $\mathbb{Z}(H)$-modules. Therefore, the lift $\Tilde{\varphi}$ asts as matrices $P_V(x)$ and $P_E(x)$ on these $\mathbb{Z}(H)$-modules. McMullen showed in \cite{mcmullen2000polynomial} that the Teichm\"uller polynomial can then be computed in term of these two matrices.

\begin{thm} 
Under the previous notations, the Teichm\"uller polynomial can be explicitly computed as follows:
    $$
    \theta_\mathcal{F}(x,u) = \frac{\det(u I - P_E(x))}{\det(u I - P_V(x))}
    $$
\end{thm}

\section{The $n$ chained links and their complements}

In \cite{kin2014dynamics}, Eiko Kin analysed in detail a $3$-manifold, known as the magic manifold. 
This manifold has the property that all the faces of its Thurston unit ball are fibered. 
She was able to precisely determine all the fibered faces and, for each integer point in a fibered face, find the topology of the associated monodromy (i.e determine its genus and the number of boundary components).
In this section we generalize the technique used for the magic $3$-manifold to study sequences of fibers in more general link complements. 
We investigate whether these $n$-chained links are fibered, what are the fibers and the associated monodromies, and the shape of Thurston unit norm ball of various $C(n,p)$'s.

\begin{defn}[$n$-chained link]
    A $n$-chained link is a link with $n$ components which are linked in a circular fashion. Some of the components may have self half-twists.
    One can always gather all such self half-twists into a single component, with same clasps shape (see figure \ref{fig:C(n,p)ex}).
    Here the word clasp designates a pair of crossings of two adjacent link components.
    Remark that there are only $2$ possible shapes of clasps and one can change one to the other by performing a self-half twist in a suitable direction.
    
    \begin{figure}[h]
        \centering
        \includegraphics[scale = 0.5]{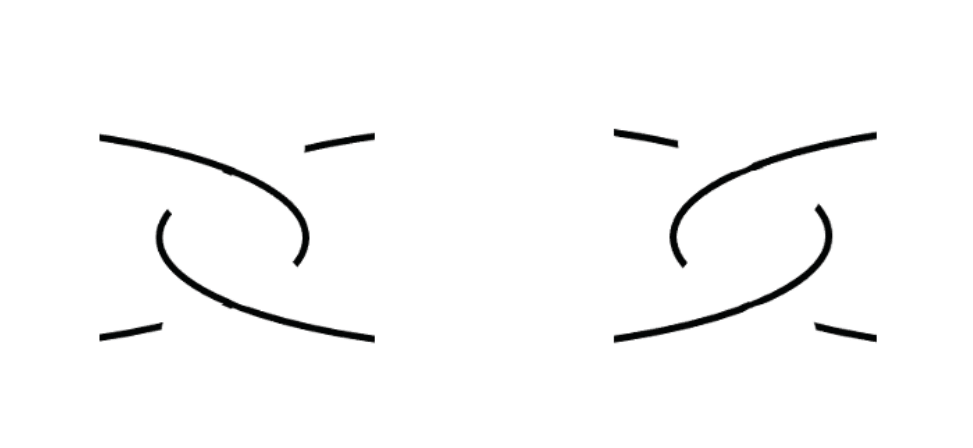}
        \caption{The $2$ different kinds of clasps, we will call the left one as a \textit{$+$ clasp} and the right one as \textit{a $-$ clasp}}
        \label{fig:clasps}
    \end{figure}
    We denote such an $n$-chained link by $C(n, p)$ where $n$ is the number of components and $p$ indicates the number of half twists.
    Note that $p$ is an integer, and the sign of $p$ indicates the direction of the half twists. We choose the positive sign as the direction of twists which makes $C(n, p)$ admits an alternating link diagram.

\end{defn}

\begin{figure}[h]
    \centering
    \includegraphics[width=.9\textwidth]{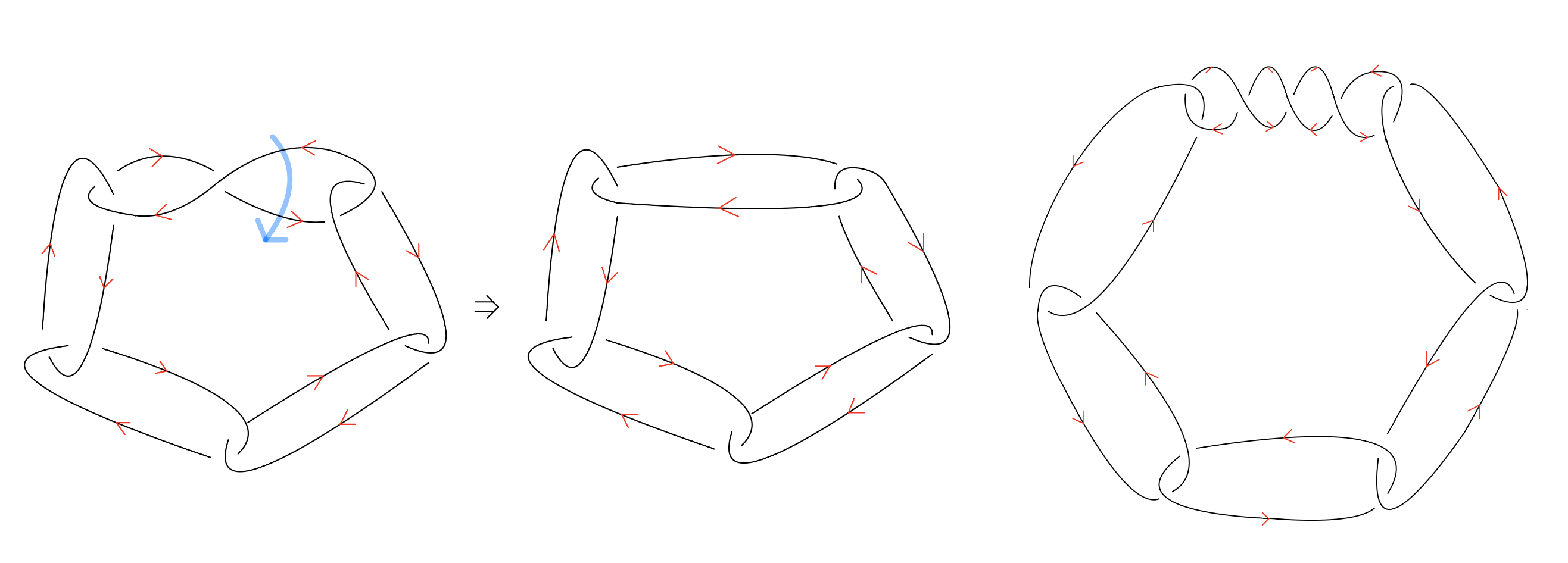}
    \caption{$C(5,-1)$ and $C(6,3)$.}
    \label{fig:C(n,p)ex}
\end{figure}

\begin{rmk}
    Note that any choice of `clasp shape' can be isotoped into the other choice.
    If we denote by $\widetilde{C}(n,p)$ the same link as $C(n,p)$ but with a different choice of a clasp shape, we can perform $-n$ half twists to flip every clasp into the other shape. 
    Keep in mind that the sign of $p$ depends on the shape of clasp.
    Since every self twist is an isotopic move, we get $\widetilde{C}(n,p) \cong C(n, -p-n)$.
    But $\widetilde{C}(n,p)$ is isomorphic to $C(n,p)$ under orientation reversing isomorphism, so we get $C(n,p) \cong C(n,-p-n)$.
      
    Usually the `minimally twisted' chain link is the one with the maximal alternating clasp shapes. 
    With our notation, the minimally twisted $n$-chained link is $C(n, -\floor{\frac{n}{2}}{})$.
\end{rmk}

Let $M(n,p)$ be the complement of a small enough neighborhood $\mathcal{N}(C(n,p))$ of $C(n,p)$.
For example, $M(3,0)$ is the magic 3-manifold.
Note that $\partial \mathcal{N}(C(n,p))$ is a disjoint union of $n$ tori.
In \cite{neumann2011arithmetic}, Neumann and Reid prove that $M(n, p)$ with $n \geq 3$ is hyperbolic if and only if $\{|n+p|, |p|\} \not\subseteq \{0, 1, 2\}$.
Moreover it is fibered, which is proven by Leininger \cite{leininger2002surgeries}.
More precisely, Leininger shows that, except for $(n,p) = (2,-1)$, $M(n,p)$ is fibered if $n \geq -p \geq 0$ by providing an actual fibers. 
We will use that construction later.

Now, we focus on the homology of $M(n,p)$.
Consider that we draw $C(n,p)$ in such a way that the top link has the $p$ half-twists, as in the figure \ref{fig:C(n,p)ex}.
We then denoted the top link component as $L_1$, and we enumerate the other components $L_2,L_3, \cdots L_n$ in a clockwise fashion. 

Each $L_i$ bounds a twice punctured disk, the punctures coming from $L_{i-1}$ and $L_{i+1}$. We will denote this twice punctured disk by $K_i$.
Note that $\{[K_i]_{1\leq i \leq n}\}$ forms a basis of $H_2(M, \partial M)$.
Then we have
\begin{lem}[\cite{baik2022topological}, lemma 4.6]\label{lem:H2coord}
    The fiber $S$ provided from \cite{leininger2002surgeries} has a coordinate $(1 ,\cdots, 1, -1)$ with respect to the base $[K_i]$.
    \emph{i.e.,} $[S] = [K_1] + \cdots + [K_{n-1}] - [K_n]$.
\end{lem}

Note that the fiber $S$ is a genus $1$ surface with $n$ boundaries and so its Euler characteristic is equal to $n$.
Our next goal is to find the fibered face $\mathcal{F}$ that contains $S$.

\section{Thurston unit ball for $C(n,0)$}

\subsection{Thurston unit ball}

We start by stating the main theorem of this section, but we will prove it only when all the needed lemmas will be established. The notation used in the main theorem will nonetheless be used throughout the whole section.

\begin{thm}\label{thm:normball0}
    The Thurston unit ball $B_n$ of $C(n,0)$ is the union of:
    \begin{enumerate}
        \item The $n$-dimensional cocube with vertices $(\pm 1, 0, \cdots, 0), \cdots, (0, \cdots, 0, \pm 1)$, and
        \item Two $n$-simplices whose vertices are $(1, 0, \cdots, 0), \cdots, (0, \cdots, 0, 1)$, $\frac{1}{n-2}(1, \cdots, 1)$ and its antipodal image.
    \end{enumerate}
    Hence the fiber $\frac{1}{n}(1,\cdots, 1, -1)$ lies in the fibered face $\mathcal{F}$ whose vertices are $(1, 0, \cdots, 0), \cdots, (0, \cdots, 1, 0), (0, \cdots, 0, -1)$ and $\frac{1}{n-2}(1, \cdots, 1)$.
    Moreover, every face of $B_n$ is a fibered face.
\end{thm}
Here the $n$-dimensional cocube is a dual of $[0,1]^n$.
One can consider it as a consecutive suspension of a closed interval $[-1,1]$.

At this moment we only note the theorem generalises the case of the magic $3$-manifold case, whose Thurston unit ball is a parallelogram with vertices $(\pm1,0,0)$, $(0,\pm1,0)$, $(0,0,\pm1)$ and $(1,1,1), (-1,-1,-1)$.
We observe that
\begin{lem}\label{lem:eqn_of_fib_face}
    Suppose $\sigma$ is the $n$-dimensional simplex of vertices $a_1, \cdots, a_n$ in a fibered face $\mathcal{F}$.
    Then $\sigma$ is a subset of a fibered face if there exists a point $a \in \sigma$ whose Thurston norm $-\chi(a)$ is equal to $1$.
    In this case the linear equation of the fibered face $\mathcal{F}$ is $\sum_{i = 1}^n x_i/a_i = 1$.
\end{lem}
\begin{proof}
    The proof is direct consequence of the fact that the Thurston unit ball is a polytope.
\end{proof}

As in the magic manifold case, we can calculate the Euler characteristic of any primitive points in $\mathcal{F}$.
\begin{cor}\label{cor:sub_facets_in_C(n,0)}
The convex hull of the points $e_1,e_2,\cdots,-e_n$ and $\frac{1}{n-2}(1,1,\cdots,1)$ is a subset of the facet $F$ of $B_n$.
Moreover, any primitive point $\alpha := (\alpha_1,\cdots,\alpha_n)$ in the cone $\mathcal{C} := \R^+\cdot\mathcal{F}$ the Euler characteristic of the representative of $\alpha$ is $\alpha_1+\cdots+\alpha_{n-1} - \alpha_{n}$.
\end{cor}
\begin{proof}
    Set
    \[
        a_i = \begin{cases}
			e_i, & 1\leq i\leq n-1\\
            -e_n, & i = n
		\end{cases}
    \]
    and $a = \frac{1}{n}(1,\cdots,1,-1)$.
    Since we already observed that $na$ is a fiber and $-\chi(na) = n$, it means that the linear equation $x_1 + \cdots + x_{n-1} - x_n = 1$ is the equation of a supporting hyperplane for the fibered face $\mathcal{F}$.
    Plugging $(\alpha_1, \cdots, \alpha_n)$ into $x_1 + \cdots + x_{n-1} - x_n$, we get the Euler characteristic for $\alpha$.
\end{proof} 

Now we are ready to prove theorem \ref{thm:normball0}.
\begin{proof}[Proof of theorem \ref{thm:normball0}]
    Note that $C(n,0)$ is circularly symmetric, so the points $p_i := \frac{1}{n}(1,\cdots,1) - \frac{2}{n}e_i$ for all $1\leq i\leq n$ is also a fiber.
    Hence by corollary \ref{cor:sub_facets_in_C(n,0)}, the $n$-dimensional parallelograms $P_i$ of vertices $e_1,\cdots,-e_i,\cdots,e_n$ and $\frac{1}{n-2}(1,1,\cdots,1)$ are subsets of the boundary of Thurston unit ball.
    (each $p_i$ is contained in $P_i$.)
    But the union of the $P_i$ forms a closed polytope, so it has to contain the boundary of Thurston unit ball.
\end{proof}

\subsection{Topological type of fibers}

Additionally to understanding the Thurston unit ball, we can also get information about each fibered surface in the fiber face $\mathcal{F}$.
To get the full topological type of representatives of given fibered points, we will use a slightly generalized version of the boundary formula proven by Kin and Takasawa, \cite{kin2008pseudo}.
\begin{lem}[\cite{baik2022topological}, lemma 4.9]\label{lem:gen_bdd_formula}
    Suppose $S$ is a minimal representative of $(a_1, \cdots, a_n) \in \mathcal{C}$.
    Then the number of boundaries of $S$ is equal to $\sum_{i = 1}^n \gcd(a_{i-1} + a_{i+1}, a_i)$, where the subscript follows $(i \mod n )+ 1$.
\end{lem}

Now we focus on the monodromy map of various fibers $S$, provided from \cite{leininger2002surgeries}.

In \cite{leininger2002surgeries}, the fiber is a sum of $n+1$ Hopf bands.
By Gabai's celebrated theorems (\cite{gabai1985murasugi}) the monodromy of our given fiber $S$ is equal to the composition of Dehn twists.
For example, the minimal representative $S$ of $(1,\cdots, 1,-1)$ in $M(n,0)$ admits a monodromy of $n$ vertical Dehn twist, which the direction of last $2$ only differs from the others, followed by $1$ horizontal Dehn twist.

\section{Thurston unit ball for $C(n, p)$ with $p > 0$}

Since the link $L = C(n,p)$ with $p>0$ is admits an alternating link diagram as we defined, its Seifert surface $S$ is the minimal genus surface for $L$.

We describe a simple algorithm to compute the genus of $C(n,p)$ using the Seifert algorithm.
\begin{thm}\label{thm:seifert_alg}
    Let $L = C(n,p)$ with $p \geq 1$.
    Given arbitrary signs on $x = (\pm 1, \cdots, \pm 1) \in H_2(M(n,p), \partial \mathcal{N}(L))$, the Thurston norm of $x$ is $n$. 
\end{thm}
\begin{proof}
    We will show this by performing the Seifert algorithm explicitly.
    Assume first that we arbitrarily fix the signs of each component of $x$.
    Note that the sign determines the orientation of each component of links.
    Let $L$ be the link with orientations corresponding to $x$.
    Imagine the link $L$ is drawn in a circular way such that the twisted one lies at the top (such as in \ref{fig:C(n,p)ex}).
    Label the twisted link as $L_1$ and continue the labeling clock-wise.
    
    We begin applying the algorithm, starting with the crossing involving $L_1$. 
    We get $p-1$ circle obtained at the half twists by Seifert algorithm, and two (unfinished) arcs on both sides of the circle.
    Now, focus on the right arc and the next link $L_2$.
    
    If signs of $L_1, L_2$ agrees, then the Seifert algorithm produces one circle and the arc is still not closed.
    Otherwise, Seifert algorithm makes the arc to be tied and ends up with a circle, and another arc will created on the right side of $L_2$.

    Inductively doing this, we get $n+p$ circles which does not depend on signs of links.
    As the number of linking in the diagram is equal to $2n+p$, we conclude the genus of $S$ is 
    \[
        \text{Genus of }S = \dfrac{2 + (2n+p) - (n+p) - n}{2} = 1
    \]
    By Gabai's theorem, it is the minimal genus surface of given link (with specific orientations we've fixed at the beginning).
    Therefore, $S$ is a minimal representative of the given $x$ and $||x|| = n$.
\end{proof}
The above lemma implies that the Thurston unit ball of $M \cong M(n,p)$ is an $n$ dimensional cocube.
\begin{cor}\label{cor:normballp}
    Thurston unit ball of $M(n,p)$ with $p \geq 1$ is an $n$ dimensional cocube with vertices $(\pm 1, 0, \cdots, 0), \cdots, (0, \cdots, 0, \pm 1)$.
\end{cor}
\begin{proof}
    Let $e_i$ be a canonical basis of $Z^n \cong H_2(M(n,p), \partial M(n,p))$.
    Since $e_i$ is represented by a $2$ punctured disk, it lies on the Thurston unit ball.
    By the above theorem, we know that $(\pm 1/n, \cdots, \pm 1/n)$ is also on the unit ball.
    For each $(\pm 1/n, \cdots, \pm 1/n)$, it is a convex combination of the canonical basis (with suitable signs).
    Therefore we conclude that the convex hull of $\{\pm e_i\}_{1\leq i\leq n}$ is exactly the Thurston unit ball.
\end{proof}

\subsection{Detecting fibered faces}
Together with theorem \ref{thm:normball0} and corollary \ref{cor:normballp}, we know the shape of Thurston unit ball of $C(n,p)$ with $p\geq 0$.
Next, we want to investigate which faces are fibered.

We denote the surface $S(n,p)_x$ to be the surface obtained from the process in the theorem \ref{thm:seifert_alg} at $C(n,p)$ with given orientation $x = (\pm 1, \cdots, \pm 1)$.
It has genus equal to $1$ and $n$ punctures.

Choose one clasp.
If the orientation of the two links at that clasp coincides, then at the clasp there are new disks and one strip. 
The $2$ half twists will connect them.
$\emph{i.e.,}$ the Hopf band is Murasugi summed at that place.

If the orientation of two links do not coincide, then choose one of link and twist half so that their orientations coincide locally at the clasp.
In this case, there is a Hopf band Murasugi summed with one half twists, as in the Figure \ref{fig:desum}.

\begin{figure}[h]
    \centering
    \includegraphics[width=.9\textwidth]{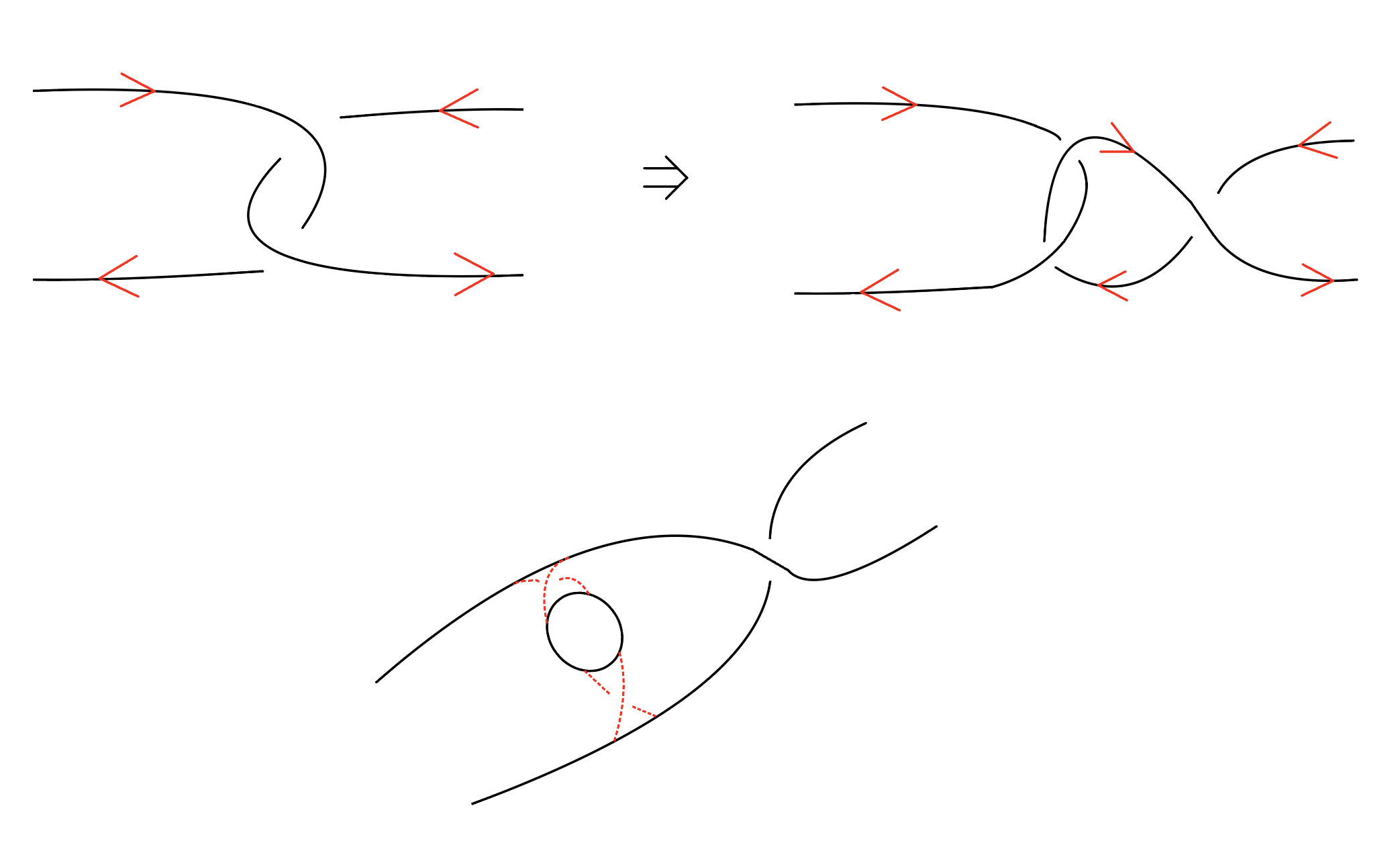}
    \caption{(Murasugi) desum the Hopf bands in the case of different orientations. Observe that the twist of result is compatible to the positive half twist with respect to the clasp shape.}
    \label{fig:desum}
\end{figure}

In the language of $H_2(M(n,p), \partial M(n,p))$, the sign change of the given orientation $x$ can be interpreted as a number of half twists after we (Murasugi) desum each vertical Hopf band. 
Also, we introduce a useful lemma in \cite{baader2016fibred}.
\begin{lem}[Example 3.1 in \cite{baader2016fibred}]\label{lem:torusknot}
    Suppose $L$ is a $(2, 2n)$-torus link with the orientation so that the Seifert surface is a full-twisted annulus. 
    Then $L$ is fibered if and only if $|n| = 1$.
    Each case is a positive/negative Hopf band.
\end{lem}
Remark that the $(2,2n)$-torus link is fibered if the orientation of the two link components is parallel.
However in our case it cannot happen since their orientations are inherited by one link component so that they must be opposite.

We summarize all the observations in the following theorems, which split into $2$ cases.
Recall that $p$ is the number of half-twists on $L_1$.
\begin{thm}[$p$ is even]\label{thm:detecting_fiber_even}
    Let $p$ is nonnegative even integer.
    For a given orientation $x = (\pm1, \pm1, \cdots, \pm1)$, denote by $s$  the number of sign changes.
    \emph{i.e.,} $s = \sum_{i = 1}^n \delta_{-1,x_ix_{i+1}}$.
    
    Then $S(n,p)_x$ is fibered if and only if $(p,s) = (0,2), (2,0)$.
\end{thm}
\begin{proof}
    As in the observation, we will desum whenever there arise a Murasugi sum of Hopf bands.
    Consequently, it remains the twisted band whose boundary is a $(2, p+s)$-torus knot. 
    By the lemma \ref{lem:torusknot}, it is fibered if and only if $p+s = 2$, it finishes the proof.
\end{proof}
\begin{thm}[$p$ is odd]\label{thm:detecting_fiber_odd}
    Let $p$ is non negative odd integer and $x, s$ is defined same as in the above.
    Then $S(n,p)_x$ is fibered if and only if $(p,s) = (1,0)$.
\end{thm}
\begin{proof}
    Only difference from the case $p$ is even is the last desumming process.
    Since there are odd number of half twists, the leftmost and rightmost part of the top link $L_1$ do not coincides.
    Hence after the desumming process, the remaining part is a twisted band whose boundary is a $(2, p+s+1)$-torus knot.
    Again by the lemma \ref{lem:torusknot}, it is fibered if and only if $p+s+1 = 2$ and $(p,s) = (1,0)$ is the only solution. 
\end{proof}

Note that $x = \frac{1}{n}(\pm1, \cdots, \pm1)$ is the barycenter of $\pm e_i$'s, the fiberedness with such orientation implies the fiberedness of faces.
Therefore, together with the theroem in \cite{leininger2002surgeries} we have the following corollary.
\begin{cor}[Fiberedness of $C(n,p)$]
    The link $C(n,p)$ is fibered if and only if $-n-2 \leq p \leq 2$.
    Moreover, every face of the Thurston unit ball of $C(n,0)$ is a fibered face. In contrast there are only $2$ fibered faces of $C(n, 1)$ and $C(n,2)$, which contain $(1,1,\cdots,1)$ and its antipodal point.
\end{cor}
\begin{proof}
    For $p = 0$, by theorem \ref{thm:detecting_fiber_even} $S(n, 0)_x$ is fibered if and only if $x$ has only one $-1$ entry and the others are all $1$ or its antipodal points.
    By corollary \ref{cor:sub_facets_in_C(n,0)}, each $x_i$ is in the distinct fibered face, hence every face of Thurston unit ball for $C(n, 0)$ is fibered.

    Suppose $p = 1$ or $p= 2$.
    By theorem \ref{thm:detecting_fiber_even} and \ref{thm:detecting_fiber_odd}, $S(n, p)_x$ is fibered if and only if $x = (1, \cdots, 1)$ or $(-1, \cdots, -1)$.
    By corollary \ref{cor:normballp}, there are only $2$ faces whose supporting planes are $\sum_{i = 1}^n x_i = \pm 1$.
\end{proof}

\section{Thurston unit ball for $C(n, p)$ with $p < 0$}

In \cite{leininger2002surgeries} Leininger proved that $C(n, -p)$ is fibered for $0 \leq p \leq n$.
In this section we investigate what is the shape of the Thurston unit ball for the complements of $n$-chained links with negative twists.
Suppose we have a $n$-chain link $C(n, -p)$ and that we have labeled each link components as before.
Note that we can untwist all the negative twists.
After resolving the negative twists on $L_1$, our link became a chain link with no twists.
However the shape of clasps may have changed during this process.
We re-assign the orientations of each $L_i$'s in a circular way after resolving all the twists on $L_1$, and, for each $i = 1,\cdots,n$, we define $e_i \in H_2(M(c, p), \partial M(c, p); \mathbb{Z})$ to be the twice punctured disk bounded and oriented by $L_i$. 
See the figure \ref{fig:C(5,-2)}.

Each link component has $2$ clasps, which may now be $+$ clasp or $-$ clasp.
Since the $-$ clasp only arise whenever one negative twist is resolved, the number of $-$ clasps is equal to $|p|$ after resolving all twists.
We define the shape vector of $C(n,p)$.
\begin{defn}
    Suppose $n \geq 4$ and $-\floor{n/2} \leq p < 0$.
    The shape vector for $C(n, p)$ is an $n$-tuple and each entry is either $+$ or $-$.
    The $i$'th entry records the shape of clasp between the link components $L_i$ and $ L_{i+1}$.
    For each $L_i$, we will say that $L_i$ has clasp shape $(\alpha, \beta)$ with $\alpha, \beta \in \{+, -\}$ if the clasp between $L_{i_1}$ and $L_i$ is an $\alpha$ shape clasp and the clasp between $L_{i}$ and $L_{i+1}$ is a $\beta$ shape clasp.
\end{defn}
Note that the number of $-$ entries in any shape vector is equal to $|p|$.

Suppose $L_i$ has $-$ shape with $L_{i-1}$ and $+$ shape with $L_{i+1}$.
Here are $2$ isotopic operations we can perform on each $L_i$.
\begin{figure}
    \centering
    \includegraphics[width = .8\textwidth]{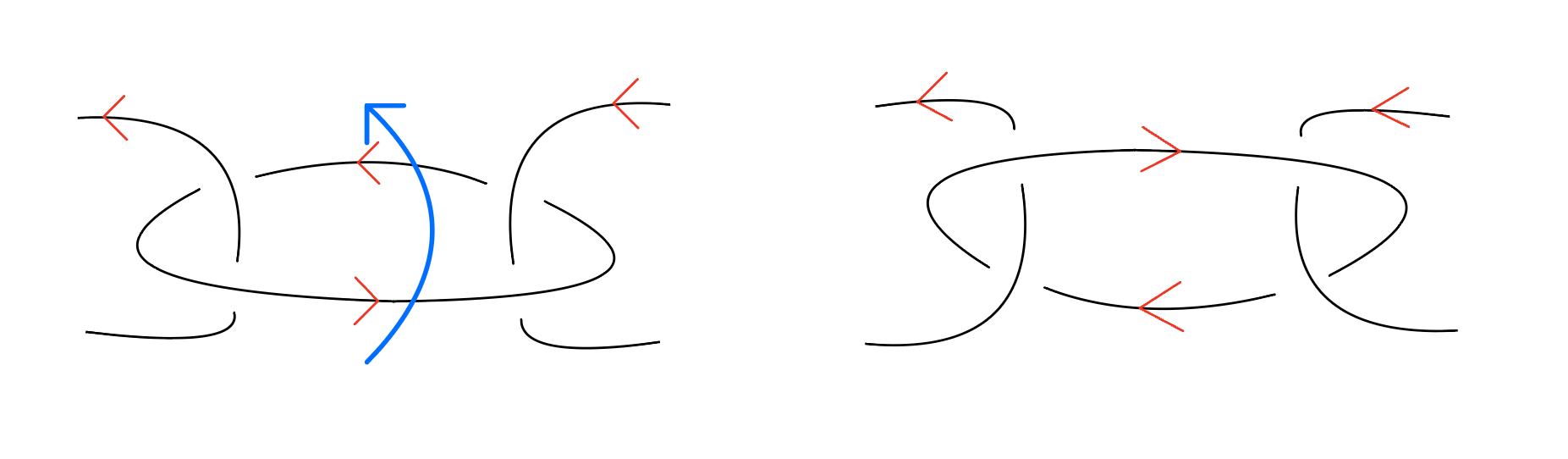}
    \caption{Flip}
    \label{fig:flip}
\end{figure}
\begin{figure}
    \centering
    \includegraphics[width = .8\textwidth]{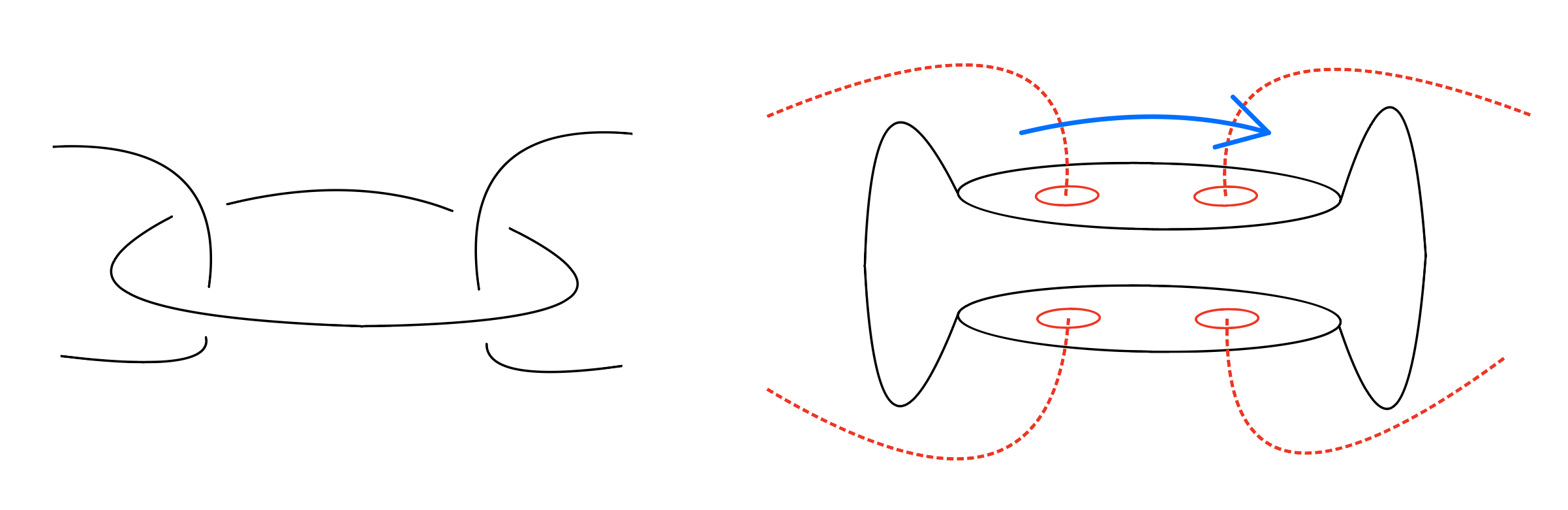}
    \caption{Full twist.}
    \label{fig:full_twist}
\end{figure}

\begin{enumerate}
    \item A \textit{flip}: we flip $L_i$ so that the $+$ clasp changes to a $-$ clasp and vice versa.
    Hence a flip exchanges the $(i-1)$'th entry and $i$'th entry of the shape vector.
    \item A \textit{full twists}:  
    There is a twice punctured disc $D$ bounded by $L_i$ whose second homology class is $e_i$.
    Cut $M(n, p)$ at $D$.
    At the slice, there are $2$ copies of $D$, say $D_1, D_2$.
    Now re-glue $D_1$ and $D_2$ after twisting either $D_1$ or $D_2$ by $360$ degrees.
\end{enumerate}

Now we have the following proposition.
\begin{prop}\label{prop:new_vertex_candidate}
    Let $n\geq 4$.
    Suppose $L_i$ admits $(+,-)$ or $(-,+)$ clasp shape. 
    Then $(\underbrace{1, \cdots, 1}_{n}) - e_i \in H_2(M(c, p), \partial M(c, p); \mathbb{Z})$ admits a $(n-1)$-punctured sphere as representative, and its Thurston norm is thus equal to $n-3$.
\end{prop}
\begin{proof}
    After performing a full twist on $L_i$, the $2$ consecutive link components $L_{i-1}$ and $L_{i+1}$ form a clasp (its shape depends on the direction of the full twist).
    If we forget about $L_i$, then the other link components became a chain link of $n-1$ components.
    Apply Seifert algorithm with all positive orientations, the Seifert surface $S$ is a $n-1$ punctured sphere.
    Since the surface $S$ does not admit $L_i$ as its boundary component, $S$ is an embedded surface in $M(n, p)$.
    Since it has no genus, this is the minimal Thurston norm representative of the homology class.
\end{proof}

Start with $C(n, -1)$.
After untwisting once, we obtain an $n$-chained link with shape vector $(-,+,\cdots,+)$.
By proposition \ref{prop:new_vertex_candidate}, $\frac{1}{n-3}(1, 0, 1, \cdots, 1)$ is a point of Thurston norm equal to $1$.
If we flip $L_2$, then the shape vector changes to $(+, -, +, \cdots, +)$.
Now we can perform a full twists on $L_3$ and then, using the same method as in the proof of proposition \ref{prop:new_vertex_candidate}, we can deduce that the point $(1,-1,0,1, \cdots,1)$ also is on the unit sphere.
Note that we have a $-1$ on the second entry this time.
Simply put, as the $0$  coordinate moves one step on the right, it also introduces a minus sign.
Hence, by this process, we get a total of $2n$ points on the unit sphere, namely
\[
    (1, 0, 1, \cdots, 1), (1, -1, 0, 1, \cdots, 1), \cdots, (1,-1, \cdots, -1, 0) \text{ and } (0, -1, \cdots, -1)
\]
and their antipodal points.

Now consider $C(n, -2)$.
The shape vector contains two $-$ entries.
We can perform a full twist on the two link components $L_i$ and $L_j$ whose clasps on both sides are different, unless $L_i$ and $L_j$ are consecutive link components.
In this case, we get $2$ zero entries in the new points and hence it represents a sphere with $n-2$ punctures.
Therefore it has Thusrton norm $n-4$. 
See the figure \ref{fig:C(5,-2)} and \ref{fig:C(5,-2)_fulltwist} also. 

\begin{figure}
    \centering
    \includegraphics[width = .4\textwidth]{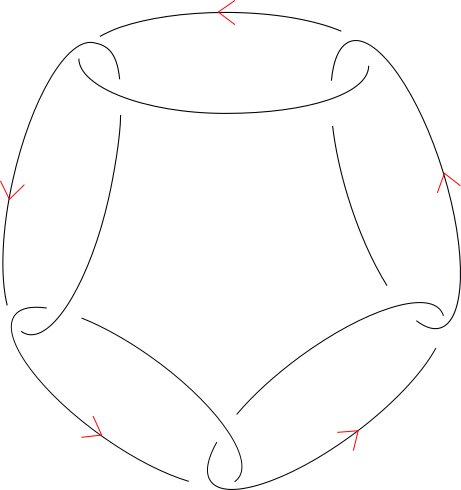}
    \caption{$C(5,-2)$. Note that the orientation of each link component is re-assigned in a circular way. Starting from the top link component, we label the components $x_1, x_2, \cdots, x_5$, clockwise.}
    \label{fig:C(5,-2)}
\end{figure}

\begin{figure}
    \centering
    \includegraphics[width = .4\textwidth]{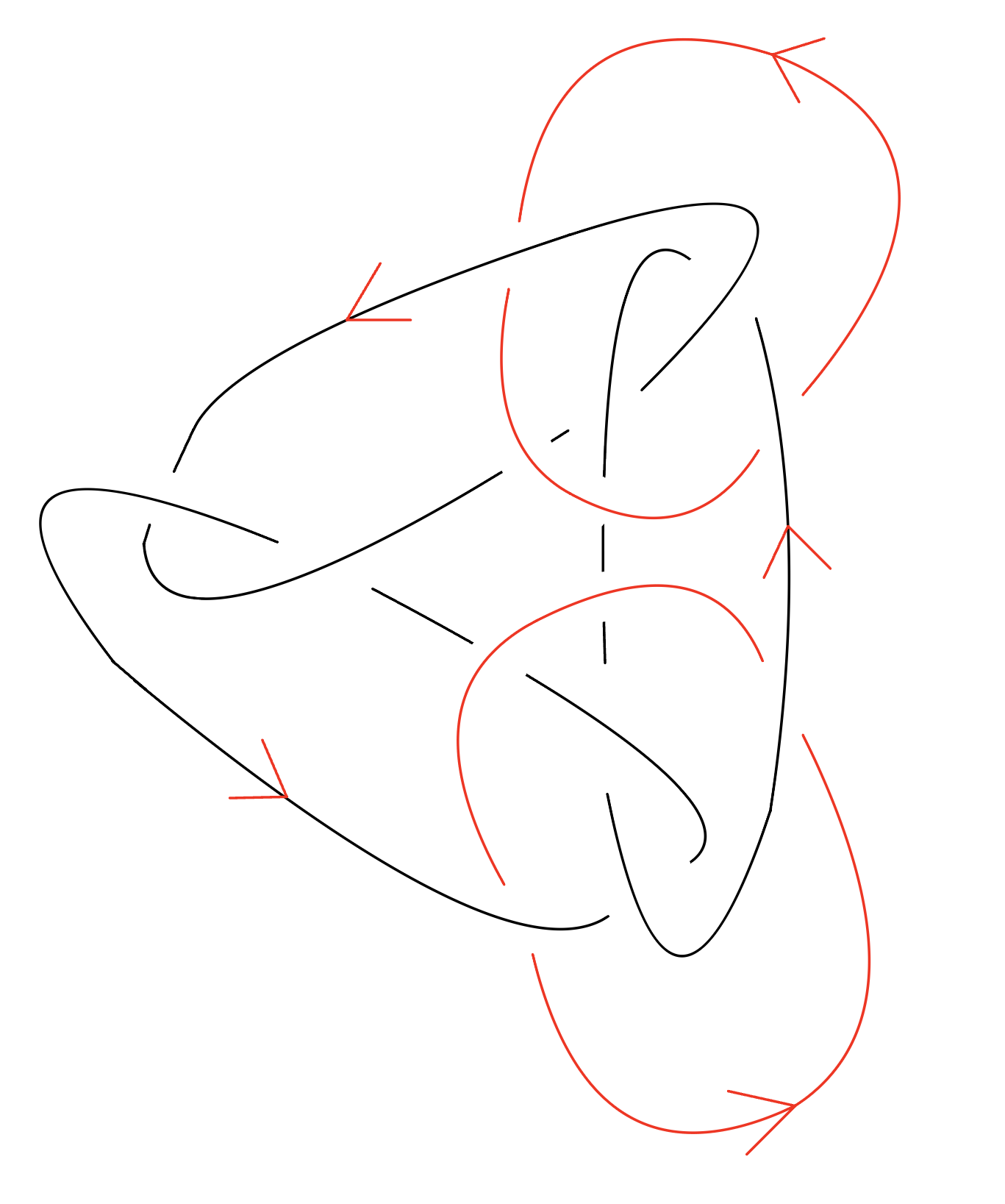}
    \caption{After two full twists, one on $L_1$ and one on $L_3$, we get the above link. Note that $(0,1,0,1,1)$ is represented by $S_{0,3}$, which is obtained by oriented sum of the $3$ disks bounded by $L_2, L_4, L_5$.}
    \label{fig:C(5,-2)_fulltwist}
\end{figure}

Therefore, we get the following corollary.
\begin{cor}\label{cor:vertex_candidates}
    By following the process described above, we obtain a set $V := V(n, p)$ of points that lie on the unit sphere of Thurston norm.
    Every points in $V$ is obtained by flipping the link components and taking full twists.
    The flip operation slides the $0$ entry to the next coordinate. 
    All points $(x_1, \cdots, x_n)$ in $V$ have the following properties.
    \begin{enumerate}
        \item There is an integer $N$ such that $N\times x_i \in \{-1, 0, 1\}$.
        Such $N$ corresponds to its Thurston norm.
        \item No two consecutive entries are equal to $0$.
    \end{enumerate}
    Let $B := B(n, p)$ be the convex hull of $V\cup \{\pm e_i\}$. 
    Then $B$ is contained in the unit Thurston norm ball.
\end{cor}

We describe the topological property of various $B(n, p)$'s.
\begin{prop}\label{prop:property_of_B(n,p)}
    Let $C(n, p)$ be a negative twisted $n$-chained link.
    Choose any $1\leq i \leq n$ and collect all points in $V(n, p) \cup \{\pm e_i\}$ with $x_i = 0$.
    Then the convex hull of such points forms an $(n-1)$-dimensional polytope and is contained in the union of $B(n-1, p+1)$ and $B(n-1, p)$.
\end{prop}
\begin{proof}
    After flipping some of link components, we can suppose $L_i$ has a clasp shape $(-,+)$.
    Perform a full twist on $L_i$ and forget $L_i$ for a moment.
    Then the remaining link components form a new link, isomorphic to $C(n, p+1)$ or $C(n, p)$, depending on the direction of the twist. 
    (More precisely, if the full twist yields a negative shape clasp between $L_{i-1}$ and $L_{i+1}$, $C(n, p+1)$ is produced and vise versa.)
    For any points in $V(n, p)$ or $V(n, p+1)$, if we plug a $0$ at the $i$'th tuple it becomes a point which lies on the unit Thurston norm ball.
    Furthermore, such point is obtained by doing only one full twist, hence it is in $V(n, p)$.
    
\end{proof}

Some faces of the polytope $B$ are actually faces of the Thurston unit ball.
We introduce another isotopic operation of for link components which have the same clasp shape on both sides.
\begin{defn}[Squeezing]
    Suppose $L_i$ has a clasp shape $(+,+)$ or $(-,-)$.
    Perform a half twists on both sides so that each clasp alters its shape.
    We will call this operation as \emph{squeezing the link} $L_i$.
\end{defn}

\begin{figure}
    \centering
    \includegraphics[width = .6\textwidth]{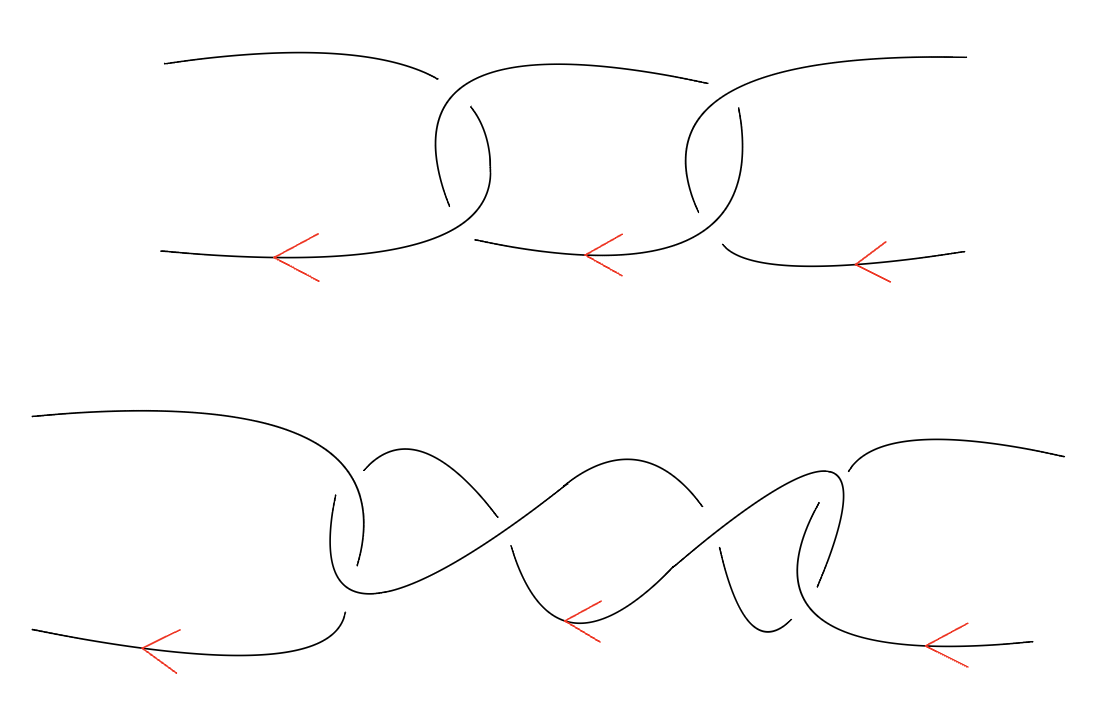}
    \caption{The link component in the middle has a $(-,-)$ shape. After squeezing, the shape of the clasp changes to $(+,+)$ with a squeezed link component.}
    \label{fig:squeezing}
\end{figure}

\begin{thm}
    Let $n \geq 4$ and $C(n, p)$ be a negative twisted $n$-chained link.
    There is a fibered face which contains the fiber obtained by squeezing one of the link components.
\end{thm}
\begin{proof}
    We will proceed by induction.
    In this proof, every full twist will be performed such that the clasp has a $+$ shape after the operation.
    \begin{enumerate}
        \item $p = -1$.
        Choose any point $q$ in $V$ and any link component $L_i$ which has a clasp shape $(+,+)$.
        There is exactly one $0$ entry in $q$. 
        Let $k$ be its index.
        By proposition \ref{prop:property_of_B(n,p)}, the slice of the unit Thurston norm ball of $C(n, p)$ at $x_k = 0$ must contain the union of $B(n-1, 0)$ and $B(n-1, -1)$.
        Choose one face in $B(n-1, 0)$.
        Since its shape vector is all $+$ (or $-$), any link component of $L_k$ has $(+,+)$ shape (or $(-,-)$).
        We choose $L_i$ except $i = k-1, k, k+1$ and squeeze it.
        Taking the inverse orientation of $L_i$, $(1, \cdots, \underbrace{-1}_{i\text{th}}, \cdots, \underbrace{0}_{k\text{th}}, \cdots, 1)$ is represented by one horizontal Hopf band Murasugi summed by $n-1$ vertical Hopf bands.
        \emph{i.e.,} $x = \frac{1}{n-1}(1, \cdots, \underbrace{-1}_{i\text{th}}, \cdots, \underbrace{0}_{k\text{th}}, \cdots, 1)$ is in the unit sphere of $C(n, -1)$.
        \newline
        Now the convex sum $\mathbf{x} := \frac{n-1}{n} \times x + \frac{1}{n} \times e_k$ is $\frac{1}{n}(1, \cdots, \underbrace{-1}_{i\text{th}}, \cdots, 1)$.
        This is still a fiber, since we choose $i$ carefully so that the squeezing still works even if we undo the full twist.
        Note that since this point is in the convex hull of $n+1$ vertices, the face containing $\mathbf{x}$ is fibered.
        \item $\floor{n/2} \leq p \leq -2$.
        By induction, we already have squeezing fibers on the face of $C(n-1, p+1)$.
        See the figure \ref{fig:squeezed_fibered_face}.
        So it remains to show that such squeezing still works after we undo the full twists.
        But since $|p+1|$ is strictly smaller than $\floor{n/2}$, there always exists a link component of shape $(+,+)$ or $(-,-)$.
        Hence by undoing full twists except near the link component, we get the fibered face which contains a squeezing fiber.
    \end{enumerate}
\end{proof}
\begin{figure}
    \centering
    \includegraphics[width = .7\textwidth]{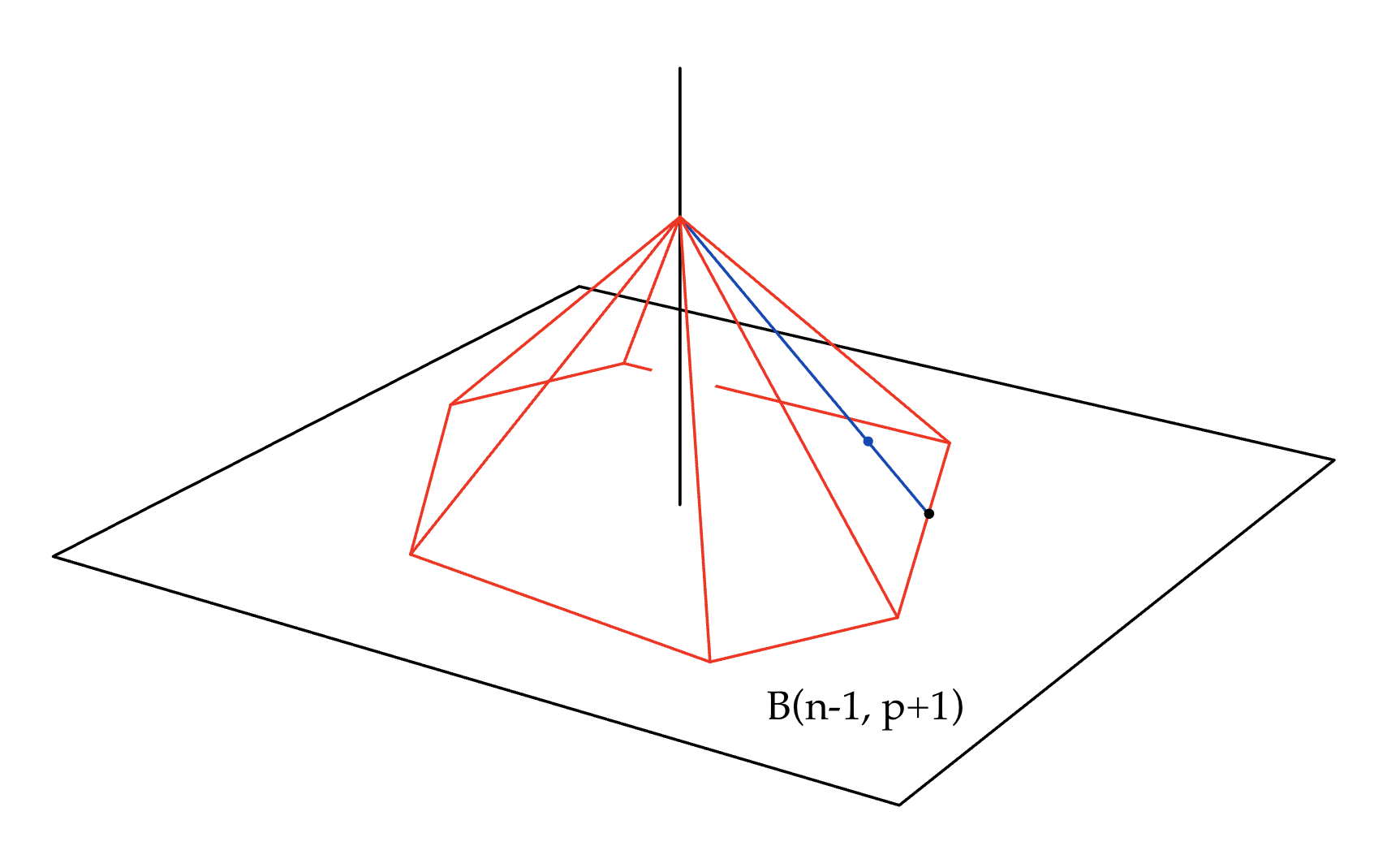}
    \caption{The vertical axis is $x_k$ coordinates, orthogonal to $\mathbb{R}^{n-1}$. The convex hull of $V(n-1, p+1)$ lies at the bottom and taking the cone with the apex $x_k = 1$. The blue dot is a fiber obtained by squeezing.}
    \label{fig:squeezed_fibered_face}
\end{figure}
The theorem implies that most of the faces in $B(n, p)$ are actually fibered faces of the Thurston unit norm ball.
We end this section with a question and some remarks.
\begin{que}
    Is $B(n,p)$ equal to the unit Thurston norm ball of $C(n, p)$ when $p<0$?
\end{que}
As far as we know, the answer is positive.
We thanks William Worden and the program Tnorm\cite{tnorm20} which helped us to calculate and verify the theorems.
We provide the table of all the vertices of the Thurson unit normal ball , calculated by Tnorm, for various $C(n, p)$'s up to $n\leq 6$ in the appendix \ref{appendix:B}.

As we already mentioned, $C(n, p)$ is fibered for $0 \leq -p \leq n$.
Assigning proper signs, Leininger's fibered surface has coordinates $(1,\cdots, 1)$ and has genus $1$ and $n$ punctures.
Since each vector $e_i$ of the canonical basis represents a twice punctured disk, we can deduce that there is a fibered face $\mathcal{F}$ which contains the standard $(n-1)$ simplex spanned by the $\{e_i\}$'s.
Furthermore, using lemma \ref{lem:eqn_of_fib_face} and similar methods as in the proof of corollary \ref{cor:sub_facets_in_C(n,0)}, we can get that the Euler characteristic of any primitive points of $(x_1, \cdots, x_n)$ with all positive entries is equal to $\sum_{i = 1}^n x_i$.
We provide some calculations for the $p<0$ case in the appendix \ref{appendix:B}.
In the remaining section, we will cover the special case of $C(n,-2)$, in which case more explicit calculations can be made.

\section{Teichm\"uller Polynomial for one fibered face of $C(n,-2)$}\label{sec:teich_polyC(n,-2)}
In this section we compute explicitly the Teichm\"uller polynomials for one fibered face of $C(n,-2)$.
Let $M_n$ be the exterior complement of the link $C(n,-2)$. 
We denote the surface obtained from performing the Seifert algorithm to $C(n,-2)$ by $S_n$.
We will sometimes omit the subscript $n$ if it is not important in the context.
Since $M_n$ is the complement of $C(n,-2)$, the second homology $H_2 = H_2(M_n, \partial M_n; \mathbb{Z})$ is a free abelian group of rank $n$, with a canonical basis given by the meridians of the link components.
With that in mind, we remark that $S_n$ is a surface of genus one with $n$ boundaries and its coordinates in $H_2$ are $(1,1,\cdots,1)$.
Since $S_n$ is a Murasugi sum of one horizontal Hopf bands with $n$ vertical Hopf bands, it is a fiber.
By Gabai's theorem, the monodromy $\varphi_n$ of this fibering is the composition of the Dehn twists along the cores of the Hopf bands.

\begin{figure}[h] 
    \centering
    \includegraphics[width=.7\linewidth]{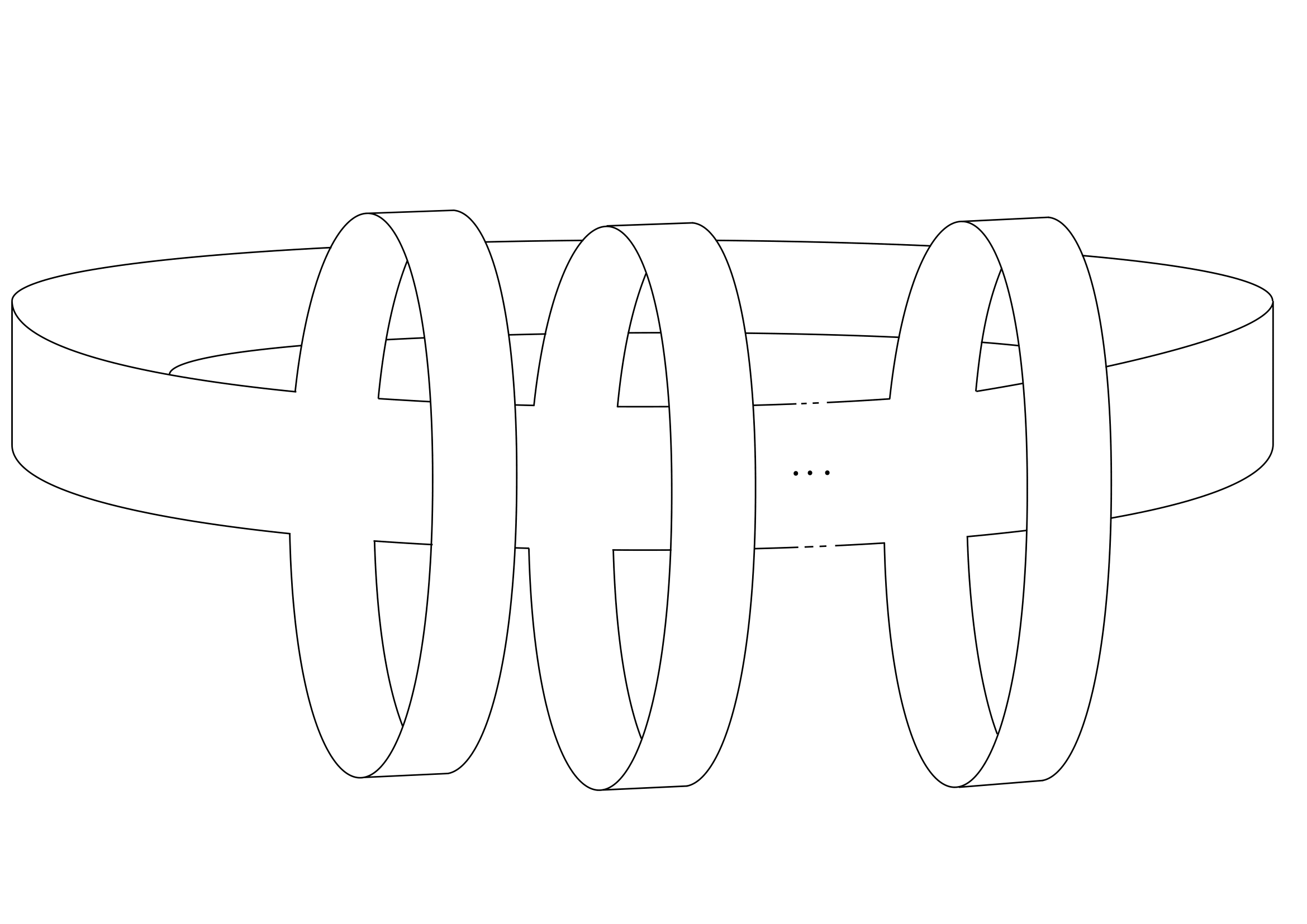}
    \caption{The Seifert surface for $C(n,-2)$, the horizontal band is a positive Hopf band and each vertical band is a negative Hopf band. Here we omit the full twists which are supposed to be at each band, as they have no role in the remainder of the calculations.}
    \label{fig:fib_of_1,C(n,-2)}
\end{figure}

Thus, if $S_n$ is placed as suggested in figure \ref{fig:fib_of_1,C(n,-2)}, the monodromy $\varphi_n$ is the composition of the $n$ vertical multi-twists directed downward followed by the left Dehn twist along the core of horizontal band.
Following the methods in \cite{billet2019teichmuller} and \cite{mcmullen2000polynomial}, we now compute the Teichm\"{u}ller polynomial corresponding to the fibered face of the Thurston unit ball which contains the fiber $(1, \cdots, 1)$.

As explained in section \ref{sec:teich}, we first need to compute $H = \Hom{(H^1(S,\mathbb{Z})^\varphi,\mathbb{Z})}$ and then understand how the lift $\Tilde{\varphi}_n$  of $\varphi_n$ acts on the cover $\Tilde{S}_n$ of $S_n$ which has $H$ as a deck transform group. 
In this case, as noted in \cite{billet2019teichmuller}, the group $H$ is simply equal to the $\varphi_n$ invariant first homology $H_1(S_n: \mathbb{Z})^{\varphi_n}$ . 
We choose $c_0,\cdots,c_n$ as a basis for $H_1(S_n; \mathbb{Z})$, where $c_0$ is the curve corresponding to the core of the horizontal band and $c_1, \cdots, c_n$ are the curves corresponding to the cores of the vertical bands, $c_1$ being the leftmost one and $c_n$ the rightmost one. 
Then $H_1(S_n: \mathbb{Z})^{\varphi_n}$ is the subspace of $H_1(S_n; \mathbb{Z})$ generated by the column vectors of

$$
B_n = 
\begin{bmatrix}
0 & 0 & \cdots & 0\\
1& 1 & \cdots & 1 \\
-1 & 0 & \cdots & 0 \\
0 & -1 & \cdots & 0\\
\vdots & \vdots & \ddots & \vdots \\ 
0 & 0 & \cdots & -1 \\
\end{bmatrix}
$$

We now need to figure out what the cover $\Tilde{S}_n$ is and how $\Tilde{\varphi}_n$ acts on it. 
Once again, the details are all given in \cite{billet2019teichmuller}, so instead of repeating them here, we will give some graphical explanation for the simplest no trivial example, that is $M_3 = C(3,-2)$. 
In this case, the cover $\Tilde{S}_n$ is explicitly drawn in figure \ref{fig:galois_cov}.

\begin{figure}[h]
    \centering
    \includegraphics[width=.7\linewidth]{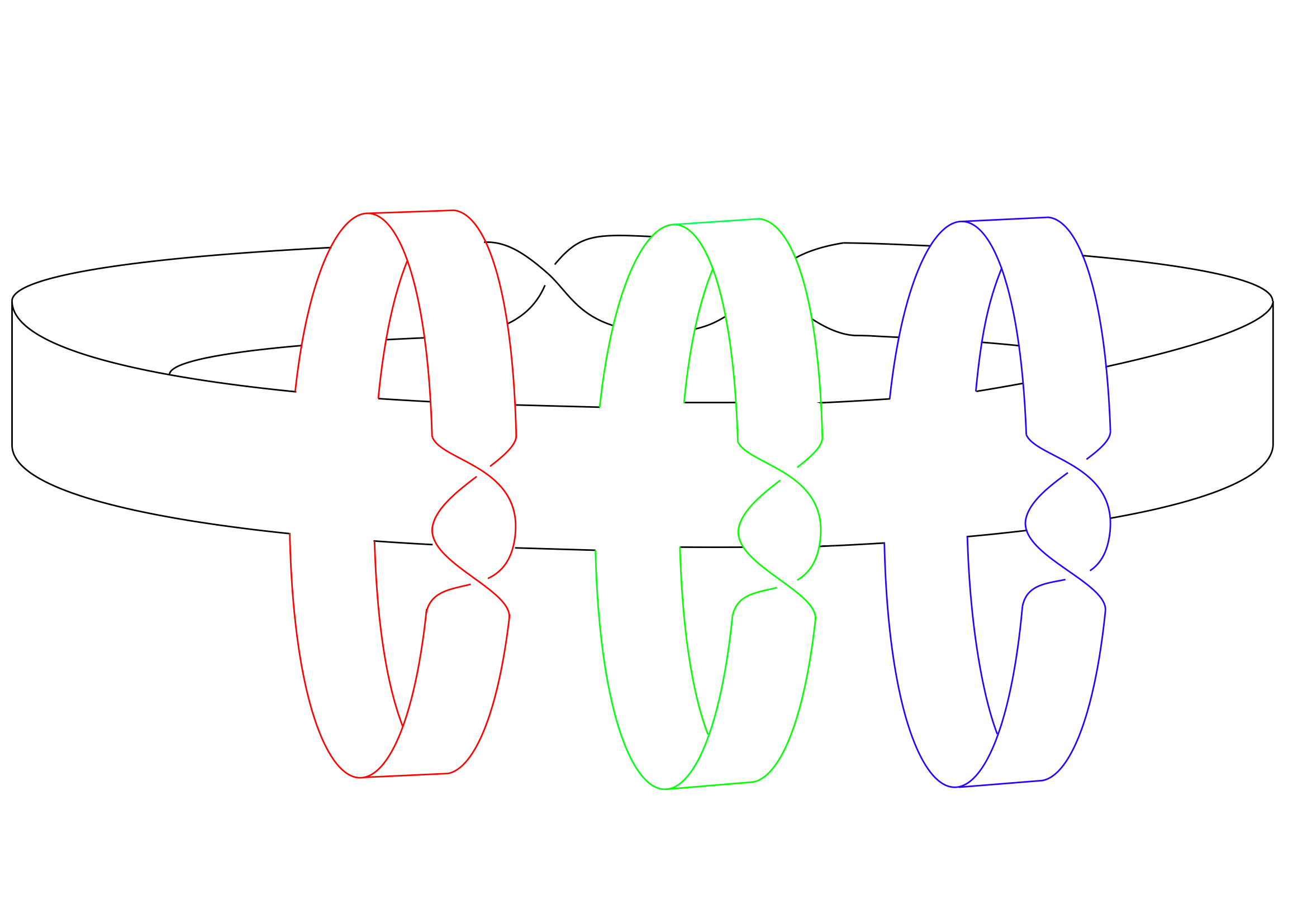}
    \caption{$S_3$, the fiber associated to the link $C(3,-2)$}
    \label{fig:S=C(4,-2)}
\end{figure}

~

\begin{figure}[h]
    \centering
    \includegraphics[width=\linewidth]{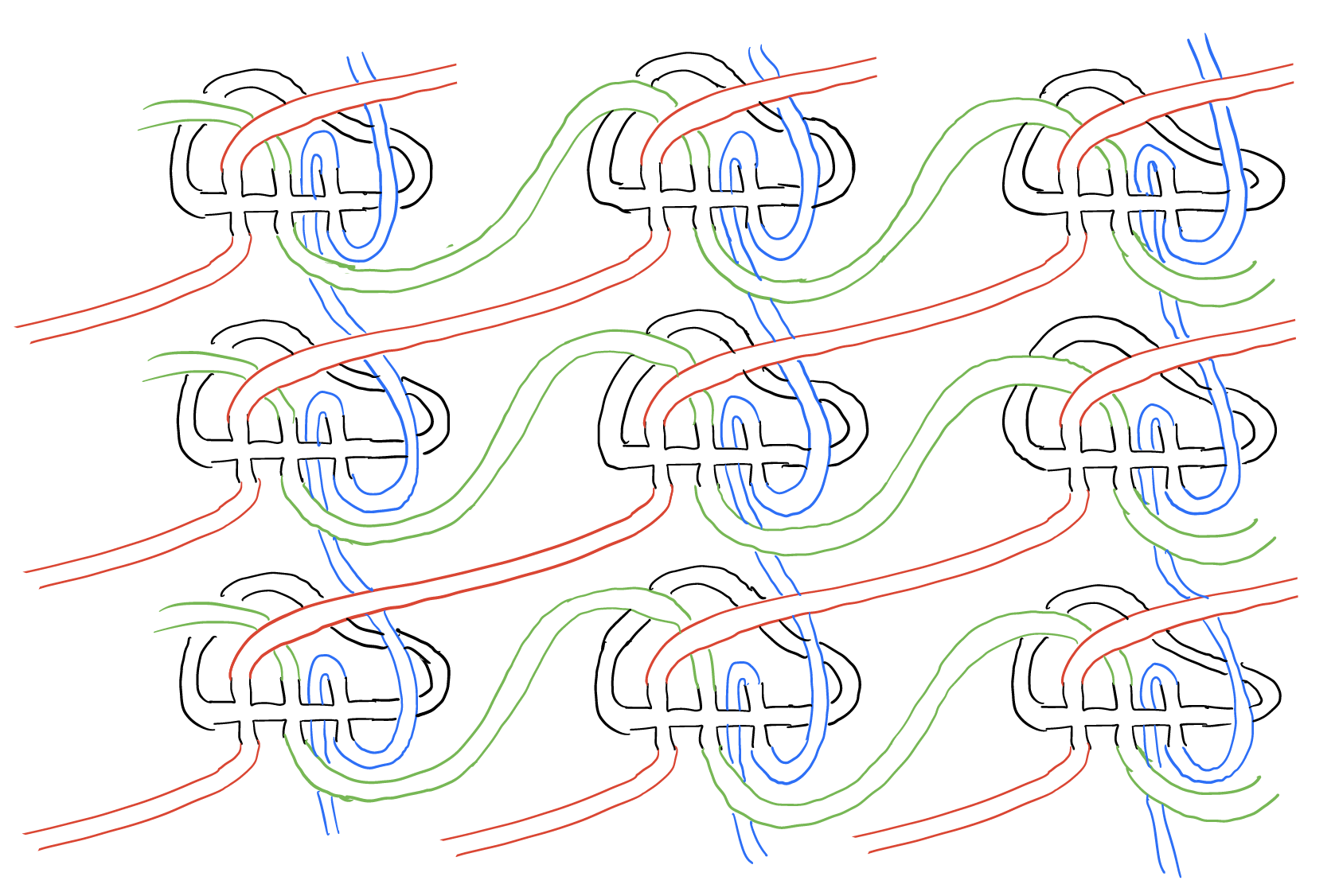}
    \caption{Galois covering $\Tilde{S}_3$ of $S_3$ with deck transform $H$}
    \label{fig:galois_cov}
\end{figure}

Let $T$ be the matrix representing the $H$-module action of $\Tilde{\varphi}_n$ on $\Tilde{S}_n$. 
The monodromy $\varphi_n$ being the composition of one horizontal Dehn twist and $n$ horizontal ones, we can decompose the matrix $T$ into $T_V$ and $T_H$. 
These matrices represent the action of the lifts of the vertical multi-twist and the horizontal Dehn twist on $\Tilde{S}$.
Note that the entries of these matrices are in $\Z[G]$, where $G$ is a deck transformation of $\widetilde{S}$, isomorphic to $\Z^{n-1}$.

Using our conventions, 
$T_V, T_H$ are $2n\times 2n$ matrices whose entries are as below; 
\[
T_V = 
\begin{bmatrix}
(x_1\cdots x_{n-1})^{-1} & 0 & \cdots & 0 & 0 & 0 & \cdots & 0\\
0 & 1 & \cdots & 0 & 0 & 0 & \cdots & 0 \\
\vdots & \vdots & \ddots & \vdots & \vdots & \vdots & \ddots & \vdots\\
0 & 0 & \cdots & (x_1\cdots x_{n-2})^{-1} & 0 & 0 & \cdots & 0 \\
1 & 0 & \cdots & 0 & 1 & 0 & \cdots & 0\\
0 & x_1^{-1} & \cdots & 0 & 0 & x_1^{-1} & \cdots & 0 \\
\vdots & \vdots & \ddots & \vdots & \vdots & \vdots & \ddots & \vdots\\
0 & 0 & \cdots & (x_1\cdots x_{n-1})^{-1} & 0 & 0 & \cdots & (x_1\cdots x_{n-1})^{-1} \\
\end{bmatrix}
\]

\[
T_H = 
\begin{bmatrix}
1 & 0 & \cdots & 0 & 1 & 1 & \cdots & 1\\
0 & 1 & \cdots & 0 & 1 & 1 & \cdots & 1 \\
\vdots & \vdots & \ddots & \vdots & \vdots & \vdots & \ddots & \vdots\\
0 & 0 & \cdots & 1 & 1 & 1 & \cdots & 1 \\
0 & 0 & \cdots & 0 & 1 & 0 & \cdots & 0\\
0 & 0 & \cdots & 0 & 0 & 1 & \cdots & 0 \\
\vdots & \vdots & \ddots & \vdots & \vdots & \vdots & \ddots & \vdots\\
0 & 0 & \cdots & 0 & 0 & 0 & \cdots & 1 \\
\end{bmatrix}
\]
We cut these matrices by 4 block matrices of $n\times n$ to abbreviate it as
\[
T_V = 
\begin{bmatrix}
D_s & 0 \\
D & D
\end{bmatrix}
\qquad 
T_H = 
\begin{bmatrix}
I & \mathbf{1} \\
0 & I
\end{bmatrix}
\]
where $D_s$ is an $n\times n$ matrix whose diagonal entries are same as $D$ shifted to the right by one and $\mathbf{1}$ is a matrix filled with $1$.

By \cite{mcmullen2000polynomial}, the Teichm\"{u}ller polynomial can be obtained by 
\[
    P(x_1, \cdots, x_{n-1}, u) := \dfrac{\det(T_VT_H - uI)}{\det(D-uI)}
\]
The remaining calculation are showed in Appendix \ref{appendix:A}. 
Let $a_k$ be the $k$th diagonal entry of $D$.
\emph{e.g.,} $a_1 = 1, a_2 = x_1^{-1}, \cdots, a_n = (x_1\cdots x_{n-1})^{-1}$.
Then we have the following formula.
\begin{thm}[Teichm\"{u}ller polynomial]\label{thm:teichpoly}
    The Teichm\"{u}ller polynomial $P$ for the fibered face $\mathcal{F}$ is
    \[
        P(x_1, \cdots, x_{n-1}, u) := A - \sum_{k = 1}^n ua_kA_k
    \]
     where $A := (a_1 - u)\cdots(a_n - u)$ and $A_k = \dfrac{A}{(a_k - u)(a_{k-1}-u)}$, subscript $k \equiv (mod~n) + 1$.
\end{thm}

The manifold $M_n$ can be viewed at the same time as a link complement and has a fibration. 

Both point of view lead to natural coordinates on $H_2 = H_2(M_n,\partial M_n ; \mathbb{Z})$. 

It is sometimes more convenient to use the coordinates coming from the llink complements point of view for the Teichm\"{u}ller polynomials. For example, if we wanted to use the Teichm\"{u}ller polynomial to compute the stretch factor of the monodromy of the fiber which has coordinates $(1,1,\cdots,1)$ in the basis given by the link components.

The Teichm\"{u}ller polynomials we computed are using the basis coming from the fibration point of view.
We thus need to find the explicit change of coordinates for going from one basis to the other.

Let us fix the notation clearly. 
The basis $Y$ given by the link complements will be denoted as $y_1,\cdots, y_n$, with $y_1$ corresponding to the link complement with the self twist. 
If the monodromy for the fibration of $M_n$ is denoted by $\varphi_n$, the corresponding basis $X$ will be $u, x_1, \cdots, x_{n-1}$ where the $x_i$ form a basis for the $\varphi_n$ invariant cohomology and $u$ corresponds to the suspension flow. 
We also let $a_0, \cdots, a_{n-1}$ be the canonical basis for $H^1(S_n, \mathbb{Z})$. 
By the computation above, we already know that $x_i = a_1 -a_{i+1}$. 
Moreover, as suggested by figure \ref{C(n,-2)_change_coordinates}, we see that $a_i = y_i -y_{i+1}$, where the indices are taken modulo $n$ as always. 
Finally, the basis element $u$ corresponding to the suspension flow is simply mapped to $y_1$.

\begin{figure}[h]
    \centering
    \includegraphics[width=1.0\linewidth]{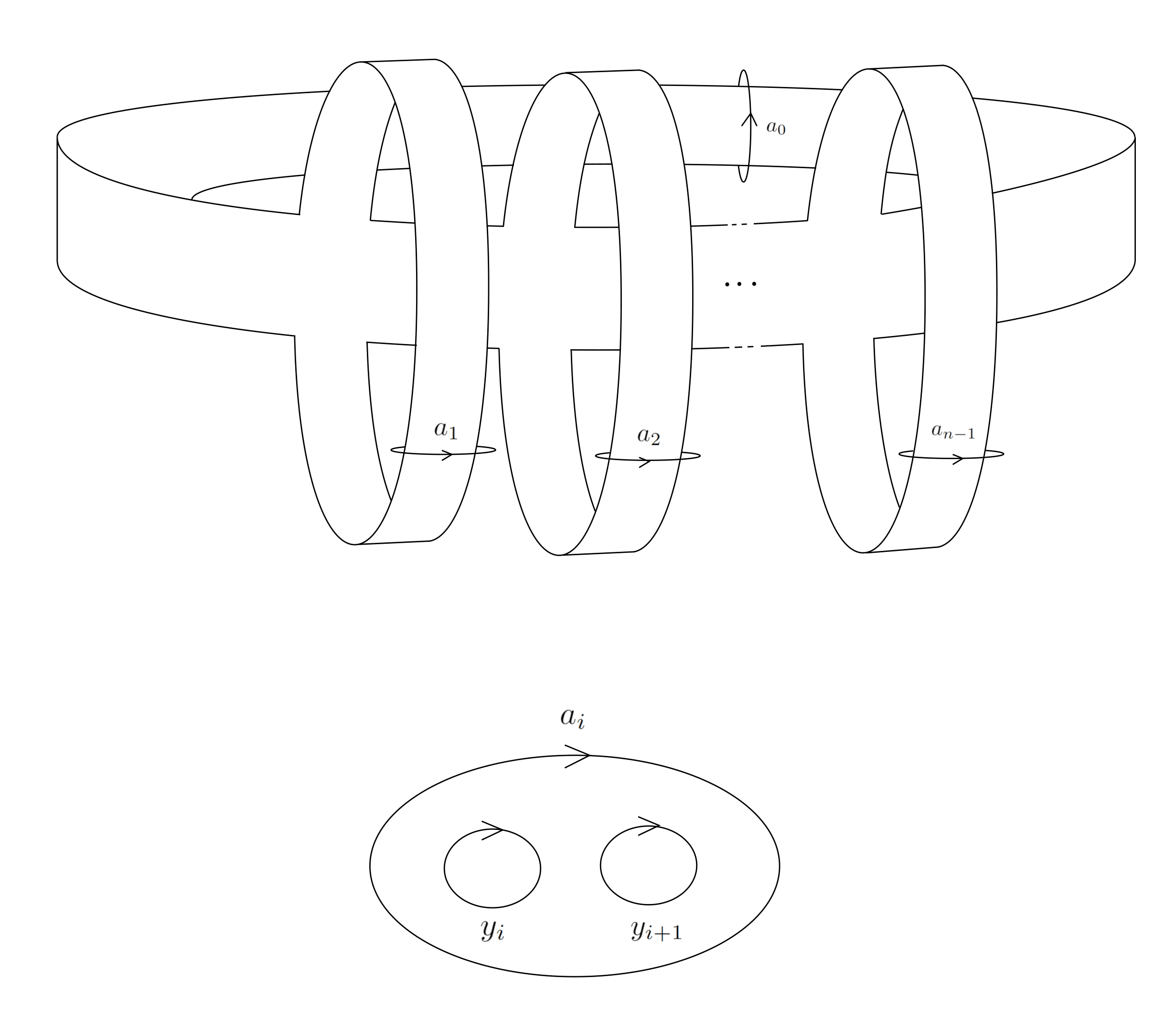}
    \caption{The surface $S_n$ with the geometric representation of the $a_i$'s. On the bottom, we see how the $a_i$'s are related to the link components, since $a_i$, $y_i$ and $y_{i+1}$ always bound a disk in $M_n$}
    \label{C(n,-2)_change_coordinates}
\end{figure}

To sum it up, the change of coordinates is given by

\begin{center}
\begin{tabular}{ c c c }
$u \to y_1$ \\
$x_1 \to y_1 - y_3$\\
$x_2 \to y_1 - y_2 +y_3 -y_4$\\
$\vdots$\\
$x_{n-2} \to y_1 - y_2 - y_{n-1} +y_n$\\
$x_{n-1} \to - y_2 +y_n$\\
\end{tabular}
\end{center}

Hence the image of the fiber whose coordinates in the basis $X$  are $(1,1, \cdots,1)$ has $(0,0,\cdots,0,1)$ as coordinates in the basis $Y$. The specialization of the Teichm\"{u}ller polynomial to the point $p = (0,0,\cdots,0,1)$ is then given by

$$
(1-u)^n -nu(1-u)^{n-2} = (1-u)^{n-2}( 1 -(n+2)u +u^2)
$$

A simple calculation shows that the largest root of this polynomial is $\frac{n +2 + \sqrt{n^2 +4n}}{2}$.

\begin{appendix}
	\section{Proof of Theorem \ref{thm:teichpoly}}\label{appendix:A}
	In this appendix we finish the calculations of the Teichm\"{u}ller polynomial of section \ref{sec:teich_polyC(n,-2)}.
	
	We need to compute the determinant of block matrices, so we make use of the following lemma.
	\begin{lem}\label{lem:blockMatrices}
		Let $M = \begin{bmatrix}
			A&B\\C&D
		\end{bmatrix}$
		a block matrix, where $A, D$ are square matrices.
		Then the determinant of $M$ is $\det(A)\det(A - BD^{-1}C) = \det(D)\det(D - CA^{-1}B)$,
		Assume further that if $C$ and $D$ commutes, then $\det(M) = \det(AD - BC)$.
	\end{lem}

	The matrix $T_VT_H - uI$ can be expressed by the block matrix as below,
	\[
	T_VT_H - uI = 
	\begin{bmatrix}
		D_s - uI & D_s\cdot \mathbf{1} \\
		D & D\cdot \mathbf{1} + D - uI
	\end{bmatrix}
	\]
	By performing some row reductions, we can simplify the matrix:
	\[
	\begin{bmatrix}
		I & 0 \\
		-DD_s^{-1} & I
	\end{bmatrix}
	\times
	\begin{bmatrix}
		D_s - uI & D_s\cdot \mathbf{1} \\
		D & D\cdot \mathbf{1} + D - uI
	\end{bmatrix}
	=
	\begin{bmatrix}
		D_s - uI & D_s\cdot \mathbf{1} \\
		uDD_s^{-1} & D - uI
	\end{bmatrix}
	\]
	Such operation does not affect the determinant and now the bottom two block matrices are both diagonals, so they commutes.
	Hence we can apply lemma \ref{lem:blockMatrices} to compute the determinant
	\begin{align*}
		\det(T_VT_H - uI) &= \det((D_s - uI)(D - uI) - uD_s\cdot \mathbf{1}\cdot DD_s^{-1})\\
		&= \det(D_s((D_s - uI)(D - uI) - u\mathbf{1}\cdot D)D_s^{-1})\\
		&= \det((D_s - uI)(D - uI) - u\mathbf{1}\cdot D)
	\end{align*}
	Let $B_k := (a_{k-1} - u)(a_k - u)$ with the indices taken module $n$.
	Then the matrix $(D_s - uI)(D - uI) - u\mathbf{1}\cdot D$ is
	\[
	\begin{bmatrix}
		B_1 - ua_1 & -ua_1 & \cdots & -ua_1 \\
		-ua_2 & B_2 - ua_2 & \cdots & -ua_2 \\
		\vdots & \vdots & \ddots & \vdots \\
		-ua_n & -ua_n & \cdots & B_n - ua_n
	\end{bmatrix}
	\]
	Note that the determinant formula for $n\times n$ matrix $[m_{ij}]$ is 
	\[
	\sum_{\sigma \in S_n} sgn(\sigma)m_{1,\sigma(1)}\cdots m_{n,\sigma(n)} 
	\]
	We sort each summand $m_{1,\sigma(1)}\cdots m_{n,\sigma(n)}$ by how many diagonal entries they contain.
	There is the only term with $n$ diagonal entries, $(B_1 - ua_1)\cdots(B_n - ua_n)$.
	Consider this as a polynomial of $u$, say $f(u)$.
	There are no terms containing $n-1$ diagonal entries.
	Now, consider the terms containing $n-2$ diagonal entries.
	If the term misses $i$ and $j$th diagonal entries, then it has to be $-u^2a_ia_j(B_1 - ua_1)\cdots(B_n - ua_n) / (B_i - ua_i)(B_j - ua_j)$.
	Note that the sign is negative, since the permutation $\sigma$ is $(i,j)$ which has a negative sign.
	Since this term offsets every term of $u^2$ in the diagonal product, after summing these terms $f(u)$ has no $u^2$ terms.
	Iterating this process, $f(u)$ only has constant and linear terms and $u^n$ term with respect to $u$.
	
	There are a lot of $u^n$ terms from the determinant formula, but it offsets each others since it is just the determinant of $u\mathbf{1}\cdot D$.
	Hence we conclude that the determinant of given matrix is 
	\[
	P'(u) = B_1\cdots B_n - u(\sum_{k = 1}^n B_1\cdots B_n / B_k)
	\]
	The formula follows from dividing $P'$ by $\det(D - uI)$.

	\section{Some calculations of $C(n,p)$ with $p<0$}\label{appendix:B}
    In this section we give our calculations and tables of vertices for some $C(n,p), p<0$'s.
    
    We thanks to William Worden and his paper \cite{cooper2021thurston} and the program called `Tnorm' which helped us verify our theorem.
    Tnorm is able to compute the vertices of the Thurston unit ball of given links complements. 
    In the tables in this section, we list the vertices, except for vertices of the form $\pm e_i$'s, together with the topological type of their representatives. So the left columns of the tables are the coordinates of the vertices and the right columns are the corresponding surfaces representing them in the second homology groups.

    \begin{center}
        \begin{tabular}{ |c|c| }
        \hline
        \multicolumn{2}{|c|}{$C(4,-1)$} \\
        \hline
        $\pm(1, 0, 1, 1)$ & $S_{0,3}$ \\
        $\pm(1, -1, 0, 1)$ & $S_{0,3}$ \\
        $\pm(1, -1, -1, 0)$ & $S_{0,3}$ \\
        $\pm(0, -1, -1, -1)$ & $S_{0,3}$ \\
        \hline
        \end{tabular} 
    \end{center}
  
    \begin{center}
        \begin{tabular}{ |c|c||c|c| }
        \hline
        \multicolumn{2}{|c||}{$C(5,-1)$} & 
        \multicolumn{2}{|c|}{$C(5,-2)$}\\
        \hline
        $\pm(1/2, 0, 1/2, 1/2, 1/2)$ & $\frac{1}{2}S_{0,4}$ & $\pm(0, 1, 0, 1, 1)$ & $S_{0,3}$\\
        $\pm(1/2, -1/2, 0, 1/2, 1/2)$ & $\frac{1}{2}S_{0,4}$ & $\pm(0, 1, -1, 0, 1)$ & $S_{0,3}$\\
        $\pm(1/2, -1/2, -1/2, 0, 1/2)$ & $\frac{1}{2}S_{0,4}$ & $\pm(1, 0, -1, 0, 1)$ & $S_{0,3}$\\
        $\pm(1/2, -1/2, -1/2, -1/2, 0)$ & $\frac{1}{2}S_{0,4}$ & $\pm(1, 0, -1, -1, 0)$ & $S_{0,3}$\\
        $\pm(0, -1/2, -1/2, -1/2, -1/2)$ & $\frac{1}{2}S_{0,4}$ & $\pm(1, -1, 0, -1, 0)$ & $S_{0,3}$\\
        \hline
        \end{tabular}    
    \end{center}

    \begin{center}
        \begin{tabular}{ |c|c| }
        \hline
        \multicolumn{2}{|c|}{$C(6,-1)$} \\
        \hline
        $\pm(1/3, 0, 1/3, 1/3, 1/3, 1/3)$ & $\frac{1}{3}S_{0,5}$\\
        $\pm(1/3, -1/3, 0, 1/3, 1/3, 1/3)$ & $\frac{1}{3}S_{0,5}$\\
        $\pm(1/3, -1/3, -1/3, 0, 1/3, 1/3)$ & $\frac{1}{3}S_{0,5}$\\
        $\pm(1/3, -1/3, -1/3, -1/3, 0, 1/3)$ & $\frac{1}{3}S_{0,5}$\\
        $\pm(1/3, -1/3, -1/3, -1/3, -1/3, 0)$ & $\frac{1}{3}S_{0,5}$\\
        $\pm(0, -1/3, -1/3, -1/3, -1/3, -1/3)$ & $\frac{1}{3}S_{0,5}$\\
        \hline
        \end{tabular} 
    \end{center}
    \begin{center}
        \begin{tabular}{ |c|c| }
        \hline
        \multicolumn{2}{|c|}{$C(6,-2)$} \\
        \hline
        $\pm(0, 1/2, 0, 1/2, 1/2, 1/2)$ & $\frac{1}{2}S_{0,4}$\\
        $\pm(1/2, 0, -1/2, 0, 1/2, 1/2)$ & $\frac{1}{2}S_{0,4}$\\
        $\pm(1/2, -1/2, 0, -1/2, 0, 1/2)$ & $\frac{1}{2}S_{0,4}$\\
        $\pm(1/2, -1/2, 1/2, 0, -1/2, 0)$ & $\frac{1}{2}S_{0,4}$\\
        $\pm(0, 1/2, -1/2, -1/2, 0, 1/2)$ & $\frac{1}{2}S_{0,4}$\\
        $\pm(1/2, 0, -1/2, -1/2, -1/2, 0)$ & $\frac{1}{2}S_{0,4}$\\
        $\pm(0, 1/2, -1/2, 0, 1/2, 1/2)$ & $\frac{1}{2}S_{0,4}$\\
        $\pm(1/2, 0, -1/2, -1/2, 0, 1/2)$ & $\frac{1}{2}S_{0,4}$\\
        $\pm(1/2, -1/2, 0, -1/2, -1/2, 0)$ & $\frac{1}{2}S_{0,4}$\\
        \hline
        \end{tabular} 
    \end{center}
    \begin{center}
        \begin{tabular}{ |c|c| }
        \hline
        \multicolumn{2}{|c|}{$C(6,-3)$} \\
        \hline
        $\pm(0, 1/2, 1/2, 0, 1/2, 1/2)$ & $\frac{1}{2}S_{0,4}$\\
        $\pm(1/2, 0, -1/2, 1/2, 0, -1/2)$ & $\frac{1}{2}S_{0,4}$\\
        $\pm(1/2, -1/2, 0, 1/2, -1/2, 0)$ & $\frac{1}{2}S_{0,4}$\\
        $\pm(0, -1/2, 1/2, 0, 1/2, -1/2)$ & $\frac{1}{2}S_{0,4}$\\
        $\pm(-1/2, 0, 1/2, 1/2, 0, -1/2)$ & $\frac{1}{2}S_{0,4}$\\
        $\pm(1/2, 1/2, 0, -1/2, -1/2, 0)$ & $\frac{1}{2}S_{0,4}$\\
        $\pm(0, 1, 0, -1, 0, 1)$ & $S_{0,3}$\\
        $\pm(-1, 0, 1, 0, 1, 0)$ & $S_{0,3}$\\
        \hline
        \end{tabular} 
    \end{center}

\end{appendix}	
\bibliographystyle{abbrv} 
\bibliography{refs}

\end{document}